\documentclass[preprint]{imsart}

\pdfoutput=1
\RequirePackage[OT1]{fontenc}
\RequirePackage{amsthm,amsmath,amssymb}
\RequirePackage{natbib}
\usepackage{graphicx}
\usepackage{multirow}
\RequirePackage[colorlinks,citecolor=blue,urlcolor=blue]{hyperref}
\usepackage[margin=1in]{geometry}

\def\T{{ \mathrm{\scriptscriptstyle T} }}
\def\R{ \mathbb{R}}
\def\Cov{ \mathrm{Cov}}
\def\Var{ \mathrm{Var}}

\begin{document}

\begin{frontmatter}

\title{Marginal Estimation of Parameter Driven Binomial Time Series Models}
\runtitle{Marginal Estimation of Binomial Time Series Mixed Models}
\author{W.T.M. Dunsmuir\\School of Mathematics and Statistics, University of New South Wales, Sydney, Australia.\\ w.dunsmuir@unsw.edu.au \\
\and \\
J.Y. He\\School of Mathematics and Statistics, University of New South Wales, Sydney, Australia\\
jieyi.he@unsw.edu.au}
\maketitle

\begin{abstract}
This paper develops asymptotic theory for estimation of parameters in regression models for binomial response time series where serial dependence is present through a latent process. Use of generalized linear model (GLM) estimating equations leads to asymptotically biased estimates of regression coefficients for binomial responses. An alternative is to use marginal likelihood, in which the variance of the latent process but not the serial dependence is accounted for. In practice this is equivalent to using generalized linear mixed model estimation procedures treating the observations as independent with a random effect on the intercept term in the regression model. We prove this method leads to consistent and asymptotically normal estimates even if there is an autocorrelated latent process. Simulations suggest that the use of marginal likelihood can lead to GLM estimates result. This problem reduces rapidly with increasing number of binomial trials at each time point but, for binary data, the chance of it can remain over $45\%$ even in very long time series. We provide a combination of theoretical and heuristic explanations for this phenomenon in terms of the properties of the regression component of the model and these can be used to guide application of the method in practice.
\end{abstract}


\end{frontmatter}

\section{Introduction}

Discrete valued time series are increasingly of practical importance with applications in diverse fields such as analysis of crime statistics, econometric modelling, high frequency financial data, animal behaviour, epidemiological assessments and disease outbreak monitoring,  and modern biology including DNA sequence analysis -- see \cite{dunsmuir2008assessing}. In this paper we focus on time series of binomial counts.

Two broad classes of models for time series of counts, based on the categorization of \cite{cox1981statistical}, are generally discussed in the literature: observation driven models, in which the serial dependence relies on previous observations and residuals; and parameter driven models, in which the serial dependence is introduced through an unobserved latent process. Estimation of parameter driven models is significantly challenging especially when the latent process is correlated. Therefore methods that provides preliminary information of the regression parameters without requiring a heavy computation load would be appealing. For example, the use of generalized linear model (GLM) estimation for obtaining estimates of the regression parameters is discussed in \cite{davis2000autocorrelation} and \cite{davis2009negative} for Poisson and negative binomial observations respectively. GLM estimation is consistent and asymptotically normal for these two types of response distribution even when there is a latent process inducing serial dependence. However as recently pointed out by \cite{wu2014parameter} and discussed in more detail below, use of GLM for binary or binomial data leads to asymptotically biased estimates. \cite{wu2014parameter} propose a semiparametric estimation method for binary response data in which the marginal probability of success modelled non-parametrically. This paper takes a different approach and suggest using estimation based on one-dimensional marginal distributions which accounts for the variance of the latent process but not the serial dependence. Such a procedure is easy to implement using standard software for fitting generalized linear mixed models (GLMM). We show that this method leads to estimates of regression parameters and the variance of the latent process, which are consistent and asymptotically normal even if the latent process includes serial dependence. Additionally the method extends easily to other response distributions such as the Poisson and negative binomial and in these cases will improve efficiency of regression parameters related to GLM estimates.

Suppose $Y_t$ represents the number of successes in $m_t$ trials observed at time $t$. Assume that there are $n$ observations $\{y_1,\ldots,y_n\}$ from the process $\{Y_t\}$ and that $x_{nt}$ is an observed $r$-dimensional vector of regressors, which may depend on the sample size $n$ to form a triangular array, and whose first component is unity for an intercept term. Then given $x_{nt}$ and a latent process $\left\{\alpha_{t}\right\}$, the $Y_{t}$ are independent with density
\begin{equation}
f_{Y_{t}}(y_{t}|x_{nt},\alpha_{t};\theta)=\exp\left\{ y_{t}W_{t}- m_{t}b(W_{t})
+ c(y_{t})\right\}  \label{eq: EFDensity}%
\end{equation}
in which
\begin{equation}
W_{t}=x_{nt}^{\T}\beta+\alpha_{t}, \label{eq: LinPredWt}%
\end{equation}
and $b(W_t)=\log(1+\exp(W_t))$ and $c(y_t)=\log \binom{m_t}{y_t}$. Then
\[
E(Y_{t}|x_{nt},\alpha_{t})= m_{t}\dot{b}(W_{t}), \quad \mathrm{Var}(Y_{t}|x_{nt},\alpha_{t})= m_{t}\ddot{b}(W_{t}).
\]
The process $\{\alpha_t\}$ is not observed and because of this is referred to as a latent process. Often $\{\alpha_t\}$ is assumed to be a stationary Gaussian linear process with zero mean and auto-covariances
\begin{equation*}
\Cov\left(\alpha_t, \alpha_{t+h}\right) = \tau R(h;\psi)
\end{equation*}
where $\tau$ is the marginal variance of $\alpha_t$ and $\psi$ are the parameters for the serial dependence in the model of $\alpha_t$. The specification of stationary Gaussian linear process covers many practical applications and we will assume that for the remainder of the paper. However, Gaussianity it not required for the main asymptotic results presented here, and in general, $\alpha_t$ can be assumed a stationary strongly mixing process. We will discuss this extension further in Section \ref{Sec: discussion}.

We let $\theta=(\beta,\tau,\psi)$ denote the collection of all parameters and let $\theta_0$ be the true parameter vector. For the above model the likelihood is defined in terms of an integral of dimension $n$ as follows,
\begin{equation}\label{eq: fullLiklihood}%
L(\theta): = \int_{\R^{n}}\prod_{t=1} ^{n}\exp\left\{y_{t}W_{t} - m_t b(W_{t})+ c(y_{t})\right\} g(\alpha;\tau,\psi) d\alpha
\end{equation}
where $g(\alpha;\tau,\psi)$ is the joint density of $\alpha=(\alpha_1,\ldots,\alpha_n)$ given the parameters $\tau$ and $\psi$.

Maximization of the likelihood \eqref{eq: fullLiklihood} is computationally expensive. Methods for estimating the high dimensional integrals in \eqref{eq: fullLiklihood} using approximations, Monte Carlo method or both are reviewed in \cite{DunsDVTS2015}. However simple to implement methods that provide asymptotically normal unbiased estimators of $\beta$ and $\tau$ without the need to fit the full likelihood are useful for construction of statistics needed to investigate the strength and form of the serial dependence. They can also provide a initial parameter values for the maximization of the full likelihood \eqref{eq: fullLiklihood}.

For practitioners, GLM estimation has strong appeal as it is easy to fit with standard software packages. GLM estimators of the regression parameters $\beta$ are obtained by treating the observations $y_t$ as being independent with $W_t=x_{nt}^{\T}\beta$ and using the GLM log-likelihood
\begin{equation}\label{eq: loglikglm}%
l_0(\beta): =\log L_0(\beta) = \sum_{t=1} ^{n}\left[y_{t}(x_{nt} ^{\T}\beta) - m_{t}b(x_{nt}^{\T}\beta) + c(y_{t})\right]
\end{equation}
We let $\tilde \beta$ denote value of $\beta$ which maximises \eqref{eq: loglikglm}. This GLM estimate assumes that there is no additional unexplained variation in the responses beyond that due to the regressors $x_{nt}$.

However, as recently noted by \cite{wu2014parameter}, GLM does not provide consistent estimates of $\beta$ when $W_t$ contains a latent autocorrelated component. To be specific, for deterministic regressors $x_{nt}=h(t/n)$ for example, $n^{-1}l_{0}(\beta)$ has limit
\begin{equation*}
Q(\beta)= \bar{m}\int_{0}^{1}\left( \int_{\mathbb{R}} \dot{b}(h(u)^{\T}\beta_0+\alpha) g(\alpha;\tau_0)d\alpha (h(u)^{\T}\beta) - b(h(u)^{\T}\beta)\right)du
+ \overset{M}{\underset{m=1}\sum} \kappa_{m} \overset{m}{\underset{j=0}\sum} \int_{0}^{1}\pi^0(j)c(j)du
\end{equation*}
where $\bar{m} = E(m_{t})$, $\kappa_{m}=P(m_{t}=m)$ and $\pi^0(j)=P(Y_t=j|x_{nt}, \theta_0)$. We show below that $\tilde\beta$ converges to $\beta'$, which maximizes $Q(\beta)$. Equivalently $\beta'$ is the unique vector that solves
\begin{equation} \label{eq: betaprime equation for 2a}
\bar{m} \int_0^1 \left(\int_{\mathbb{R}} \dot{b}(h(u)^{\T}\beta_0 + \alpha)g(\alpha;\tau_0)d\alpha - \dot{b}(h(u)^{\T}\beta^\prime)\right)h(u) du = 0
\end{equation}

In the Poisson or negative binomial cases, $m_t\equiv 1$, and $E(Y_t)=E(\dot{b}(x_{nt}^{\T}\beta_0+\alpha_t))=\dot{b}(x_{nt}^T\beta_0 + \frac{\tau}{2})$, in which $\tau/2$ only modifies the regression intercept but does not influence the response to other regression terms. Such an identity does not usually hold for binomial observations. When $\tau_{0}>0$, the relationship between $\beta'$ and $\beta_0$ in binomial logit regression models has been investigated by several researchers. For example, \cite{neuhaus1991comparison} proved that the logit of the marginal probability $\int (1+e^{-(x^T\beta_0+\alpha)})^{-1}g(\alpha)d\alpha$ can be approximated with $x^T\beta^\ast$, where $\vert\beta^\ast\vert \le \vert\beta_0\vert$ for single covariate $x$ and the equality is only attained when $\tau=0$ or $\beta_0=0$. \cite{wang2003matching} proved that only if $g(\cdot)$ is the ``bridge" distribution, the logit of $\int (1+e^{-\alpha-x^T\beta_0})^{-1} g(\alpha)d\alpha$ equals to $x^T\beta_0$ holds; \cite{wu2014parameter} proposed their MGLM method because GLM\ estimates for binomial observations generated under model \eqref{eq: LinPredWt} are inconsistent.

To overcome the inconsistency observed in GLM\ estimation, in this paper we propose use of marginal likelihood estimation, which maximises the likelihood constructed under the assumptions that the process $\alpha_t$ consists of independent identically distributed random variables. Under this assumption the full likelihood \eqref{eq: fullLiklihood} is replaced by the ``marginal" likelihood
\begin{equation}
L_1(\delta) = \prod_{t=1}^{n}f(y_{t}|x_{t}, \delta) = \prod_{t=1}^{n} \int_{\R} \exp\left( y_{t}W_{t}-m_t b(W_{t}) + c(y_{t})\right) g(\alpha_t;\tau) d\alpha_t. \label{eq: marginal likelihood}%
\end{equation}
and the corresponding ``marginal" log-likelihood function is
\begin{equation}
l_1(\delta) = \sum_{t=1}^{n}\log f(y_{t}|x_{t},\delta)=\sum_{t=1} ^{n}\log \int_{\R} \exp\left( y_{t}W_{t}- m_tb(W_{t}) + c(y_{t}) \right) g(\alpha_t;\tau) d\alpha_t.  \label{eq: marginal log likelihood}%
\end{equation}
where $\delta = (\beta,\tau)$ and $g(\cdot,\tau)$ is the density for a mean zero variance $\tau$ normal random variable. Let $\hat\delta$ be the estimates obtained by maximising
\eqref{eq: marginal log likelihood} over the compact parameter space $\Theta:=\{\beta \in \R^r: \|\beta - \beta_0\| \le d_1\}\bigcap \{\tau\ge 0: |\tau - \tau_0| \le d_2\}$, where $d_1<\infty$, $d_2<\infty$.

Marginal likelihood estimators of $\hat\delta$ can be easily obtained with standard software packages for fitting generalized linear mixed models. Since these marginal likelihood estimates $\hat\delta$ are consistent, they can be used as the starting value of full likelihood based on \eqref{eq: fullLiklihood}. Additionally, the asymptotic distribution of $\hat\delta$, and the standard deviation derived from the asymptotic covariance matrix can be used to assess the significance of regression parameters $\hat\beta$. Moreover, in another paper we have developed a two-step score-type test to first detect the existence of a latent process and if present whether there is serial dependence. The asymptotic results of this paper are needed in order to derive the large sample chi-squared distribution of the second step, the score test for detecting serial dependence.

Large sample properties of the marginal likelihood estimates $\hat\delta$ are provided in Section \ref{Sec: Asympototic Theory for Marginal Likelihood Estimates}; The simulations of Section \ref{Sec: Simulations} show that marginal likelihood estimates lead to a high probability of $\hat\tau=0$ when the number of trials, $m_t$, is small. In particular $P(\hat\tau=0)$ can be almost 50\% for binary data. Hence Section \ref{Sec: Pile up probability} focuses on obtaining asymptotic approximations to the upper bound for $P(\hat\tau=0)$, which is useful to quantify the proportion of times the marginal likelihood procedures `degenerates' to the GLM procedure. Also in Section \ref{Sec: Pile up probability} we derive a theoretical mixture distribution which provided better approximation in this situation. Section \ref{Sec: Simulations} presents simulation evidence to demonstrate the accuracy of the asymptotic theory and the covariance matrix of the marginal likelihood estimate. Section \ref{Sec: MGLM est} discusses the difference between marginal likelihood estimation and MGLM estimation of \cite{wu2014parameter}. Section \ref{Sec: discussion} concludes.

\section{Asymptotic Theory for Marginal Likelihood Estimates} \label{Sec: Asympototic Theory for Marginal Likelihood Estimates}

We present the large sample properties for the marginal likelihood estimates of $\beta$ and $\tau$ obtained by maximizing \eqref{eq: marginal log likelihood}. We begin by presenting the required conditions on the latent process $\{\alpha_{t}\}$, the regressors $\{x_{nt}\}$ and the sequence of binomial trials $\{m_{t}\}$.

A process $\{\alpha_t\}$ is strongly mixing if
\[
\nu(h)=\sup_{t} \sup_{A\in \emph{F}_{-\infty}^{t}, B\in \emph{F}_{t+h}^\infty}|P(AB)-P(A)P(B)|\to 0
\]
as $h\to \infty$, where $\emph{F}_{-\infty}^{t}$ and $\emph{F}_{t+h}^\infty$ are the $\sigma$-fields generated by $\{\alpha_s, s\leq t\}$ and $\{\alpha_s, s\geq t+h\}$ respectively.

In practice, the number of trials $m_{t}$ may vary with time. To allow for this we introduce:
\newtheorem{cond}{Condition} \begin{cond}\label{Cond: mt}
The sequence of trials $\{m_t: 1\le m_t\le M\}$ is a stationary strongly mixing process independent to $\{X_t\}$; the mixing coefficients satisfy: $\overset{\infty}{\underset{h=0}\sum}\nu(h)< \infty$. Let $\kappa_j=P(m_t=j)$, assume $\kappa_M>0$, $\overset{M}{\underset{j=1}\sum} \kappa_j=1$.
\end{cond}
\noindent An alternative would be to take $m_t$ as deterministic and asymptotically stationary in which case the $\kappa_j$ would be limits of finite sample frequencies of occurrences $m_t = j$. Both specifications obviously include the case where $m_t=M$ for all $t$, of which $M=1$ yields binary
responses.

As in previous literature \cite{davis2000autocorrelation}, \cite{davis2009negative} and \cite{wu2014parameter} we allow for both deterministic and stochastic regressors:
\begin{cond}\label{Cond: Reg Trend Type}
The regression sequence is specified in one of two ways:
\begin{description}
		\item (a) Deterministic covariates defined with functions:  $x_{nt}=h(t/n)$ for some specified piecewise continuous vector function $h: [0,1]\to \R^r$.
		
		\item (b) Stochastic covariates which are a stationary vector process: $x_{nt}=x_t$ for all $n$ where $\{x_t\}$ is an observed trajectory of a stationary process for which $E(e^{s^T X_t}) <\infty$ for all $s\in \mathbb{R}^r$.
	\end{description}
\end{cond}

\begin{cond}\label{Cond: Reg Full Rank}
Let $r=\dim(\beta)$. The regressor space $\mathbb{X}=\{x_{nt}: t\ge 1\}$, assume $\texttt{rank}(\texttt{span}(\mathbb{X}))=r$.
\end{cond}
The full rank of the space spanned by the regressors required for Condition \ref{Cond: Reg Full Rank} holds for many examples. For instance, for deterministic regressors generated by functions given in Condition 2a, such $X_i$, $i=1,\ldots,r$ exist if there are $r$ different values of $u_i= (u_1,\ldots, u_r)$ such that the corresponding function $(h(u_1),\ldots,h(u_r))$ are linearly independent. For stochastic regressors generated with a stationary process given in Condition 2b, linearly independent $X_i$, $i=1,\ldots,r$ can be found almost surely if $\mathrm{Cov}(X)>0$.

\begin{cond} \label{Cond: Mixing} The latent process $\{\alpha_t\}$, is strictly stationary, Gaussian and strongly mixing with the mixing coefficients satisfying  $\sum_{h=0}^\infty \nu(h)^{\lambda/(2+\lambda)} < \infty$ for some $\lambda>0$.
\end{cond}

Conditions for a unique asymptotic limit of the marginal likelihood estimators are also required. Denote the marginal probability of $j$ successes in $m_t$ trials at time $t$ as
\begin{equation*}
\pi_{t}(j)=\int_{\R} e^{jW_t - m_tb(W_t)+c(j)}\phi(z_t) dz_t, \quad j=1,\ldots,m_t.
\end{equation*}
where $W_t=x_{nt}^T\beta + \tau^{1/2} z_t$, and $z_t = \alpha_t/\tau^{1/2}$ has unit variance. If $\{\alpha_t\}$ is Gaussian, so is the process $\{z_t\}$ and $z_t\sim N(0,1)$ with density function $\phi(\cdot)$. Similarly let $\pi^0_{t}(j)$ be the marginal probability evaluated with the true values $\beta_0$ and $\tau_0$ at time $t$. Define

\begin{equation} \label{eqn: Qn delta}
Q_n(\delta)=\frac{1}{n} l_1(\delta)
\end{equation}
conditional on $m_t$ and $x_{nt}$,

\begin{equation}\label{eqn: E Qn delta}
E(Q_n(\delta))=\frac{1}{n}\sum_{t=1}^n \sum_{j=0}^{m_t}\pi^0_{t}(j)\log \pi_{t}(j), \quad \delta\in \Theta.
\end{equation}

Under Conditions \ref{Cond: mt} and \ref{Cond: Reg Trend Type}, $E(Q_n(\delta))\overset{a.s.}\to Q(\delta)$. Let $\pi(j,\cdot)=P(Y=j|\cdot, \delta)$, and $\pi^0(j,\cdot)$ is evaluated with $\delta_0$, then under Condition 2a,
\begin{equation} \label{eqn: lim Q delta Cond 2a}
Q(\delta) = \sum_{m=1}^M \kappa_m \int_0^1 \sum_{j=0}^m \pi^0(j,h(u))\log \pi(j,h(u)) du
\end{equation}
and, under Condition 2b,
\begin{equation}\label{eqn: lim Q delta Cond 2b}
Q(\delta) = \sum_{m=1}^M \kappa_m \int_{\mathbb{R}^r} \sum_{j=0}^m \pi^0(j,x)\log \pi(j,x) dF(x)
\end{equation}
the proof is included in the proof of Theorem \ref{Thm: Consist&Asym Marginal Like Estimator}.

\begin{cond}\label{Cond: Ident and Const}
$Q(\delta)$ has a unique maximum at $\delta_0 = (\beta_0, \tau_0)$, the true value.
\end{cond}

We now establish the consistency and asymptotic normality of the marginal likelihood estimator.
\newtheorem{thm}{Theorem} \begin{thm}[Consistency and asymptotic normality of marginal likelihood estimators] \label{Thm: Consist&Asym Marginal Like Estimator}%
Assume $\tau_{0}>0$ and Conditions \ref{Cond: mt} to \ref{Cond: Ident and Const}, then $\hat \delta \overset{\textrm{a.s.}}{\to} \delta_0$ and $\sqrt{n}(\hat\delta - \delta_0)
\overset{d}\rightarrow N(0, \Omega_{1,1}^{-1}\Omega_{1,2}\Omega_{1,1}^{-1})$
as $n \to \infty$, in which
\begin{equation}\label{eq: PD InfMat}%
\Omega_{1,1} = \underset{n\to \infty}\lim \frac{1}{n}\sum_{t=1}^{n}E(\dot{l}_{t}(\delta_0) \dot{l}_{t}^{\T}(\delta_0)) > 0
\end{equation}
\begin{equation}\label{eq: PD CovMat}%
\Omega_{1,2}=\underset{n\to \infty}\lim \frac{1}{n}\sum_{t=1}^{n}\sum_{s=1}^{n} \mathrm{Cov} (\dot{l}_{t}(\delta_0), \dot{l}_{s}(\delta_0))
\end{equation}
where
\begin{equation}\label{eq: deriv1 loglik t}%
\dot{l}_{t}(\delta_0) = \frac{\partial \log \pi_{t}(y_{t})}{\partial\delta}\vert_{\delta_0}= f^{-1}(y_{t}|x_{nt},\delta_0)\int (y_{t}- m_{t}\dot{b}(x_{nt}^{\T}\beta_0+\tau_{0}^{1/2}z_{t})) \binom{x_{nt}}{\frac{z_t}{2\sqrt{\tau_{0}}}}f(y_{t}|x_{nt},z_t,\delta_0)\phi(z_t)dz_t
\end{equation}
\end{thm}

To use this theorem in practice requires at least that the identifiability condition holds and
that the covariance be estimated from a single series. We address these aspects in detail
in Section \ref{SSc: Ident} and \ref{SSc: CovMat Est}. In addition, particularly for binary responses, marginal likelihood estimators produce a high probability of $\hat\tau=0$. We address
this in detail in Section \ref{Sec: Pile up probability}, where we propose an improved asymptotic distribution based on a mixture.

\subsection{Asymptotic identifiability} \label{SSc: Ident}%

We now discuss circumstances under which Condition \ref{Cond: Ident and Const} holds. Now for any $\delta\in \Theta$, $Q(\delta)\le Q(\delta_0)$, since for any $x$, $\sum_{j=0}^{m}\pi^0(j,x)\log \pi(j,x) \le \sum_{j=0}^{m}\pi^0(j,x)\log \pi^0(j,x)$. Thus the model is identifiable if and only if for any $\delta \in \Theta$, $Q(\delta)- Q(\delta_0)< 0$ if $\delta \ne \delta_0$.

\newtheorem{lem}{Lemma} \begin{lem}\label{lem: Identifiable Binom}
Assume $M\ge 2$ and Condition \ref{Cond: Reg Full Rank}, then Condition \ref{Cond: Ident and Const} holds for marginal likelihood \eqref{eq: marginal log likelihood}.
\end{lem}
\noindent{\small The proof is outlined in Appendix A.}

For binary data, $M=1$, then $Q(\delta)=Q(\delta_0)$ if $\pi(1,x) = \pi^0(1,x)$, $\forall x\in \mathbb{X}$. Hence model \eqref{eq: marginal log likelihood} is not identifiable if $\exists \delta\ne \delta_0$ such that $\pi(1) = \pi^0(1)$ everywhere on $\mathbb{X}$, that is,  for each distinct value $X_i \in\mathbb{X}$, such $(\beta,\tau)\ne (\beta_0,\tau_0)$ can be found to establish
\begin{equation}\label{eq: pi1 to pi01}%
\pi(1,X_i)= \int \frac{e^{X_i^T\beta+\sqrt{\tau} z}}{1+e^{X_i^T\beta +\sqrt{\tau} z}}\phi(z)dz =
\int \frac{e^{X_i^T\beta_0 +\sqrt{\tau_0} z}}{1+e^{X_i^T\beta_0+\sqrt{\tau_0} z}}\phi(z)dz = \pi^0(1,X_i).
\end{equation}

If $\tau=\tau_0$, then \eqref{eq: pi1 to pi01} implies $X_i^{\T}\beta=X_i^{\T}\beta_0$. Under Condition \ref{Cond: Reg Full Rank}, $r$ linearly independent $X_i$, $\mathrm{X}=(X_1,\ldots,X_r)$ can be found on $\mathbb{X}$ to establish $\mathrm{X}^{\T}\beta=\mathrm{X}^{\T}\beta_0$. Then $(\beta-\beta_0)$ has a unique solution of $\textbf{0}_r$. Hence if $\tau=\tau_0$, \eqref{eq: pi1 to pi01} holds if and only if $\beta=\beta_0$, and Condition
\ref{Cond: Ident and Const} holds.

If $\tau\ne \tau_0$, for each $X_i$, a unique solution of $a_i$, $a_i = X_i^{\T}\beta \ne X_i^{\T}\beta_0$ can be found for \eqref{eq: pi1 to pi01}. Assume the regressor space $\mathbb{X}$ is a set of discrete vectors such that $\mathbb{X}=\{X_i: 1\le i\le L\}$, where $X_i\ne X_j$ if $i\ne j$. Let $\mathrm{X}=(X_1,\ldots,X_L)$ be a $r\times L$ matrix. Then \eqref{eq: pi1 to pi01} holds for each $X_i\in\mathbb{X}$ if there exists such solution of $\beta$ that $\mathrm{X}^T\beta=\mathrm{A}$, $\mathrm{A}=(a_1, \ldots, a_L)$. Since $\tau\ne \tau_0$, $\beta=\beta_0$ is not excluded from the possible solutions of $\beta$. If $L=r$, a unique solution of $\beta$ exists, hence there exists such $\delta\ne \delta_0$ that establishes \eqref{eq: pi1 to pi01} for all $X_i\in \mathbb{X}$, therefore \eqref{eq: marginal log likelihood} is not identifiable. If $L>r$, note $\texttt{rank}(\mathrm{X})=r$, then $\mathrm{X}^{\T}\beta=\mathrm{A}$ is overdetermined. Thus a solution of $\beta$ does not always exist and in these situations Condition \ref{Cond: Ident and Const} holds, however a general proof without further conditions on the regressors is difficult. Instead, we provide a rigourous proof to show Condition \ref{Cond: Ident and Const} holds for binary data when the regressor space $\mathbb{X}$ is connected.
\begin{lem}\label{lem: Identifiable Binary}
Let $M=1$. In addition to Condition \ref{Cond: Reg Full Rank}, $\mathbb{X}$ is assumed to be a connected subspace of $\mathbb{R}^r$, then Condition \ref{Cond: Ident and Const} holds.
\end{lem}
\noindent{\small Proof: see the appendix A}.

\subsection{Estimation of the Covariance matrix}\label{SSc: CovMat Est}

To use Theorem \ref{Thm: Consist&Asym Marginal Like Estimator} the asymptotic covariance matrix $\Omega_{1,1}^{-1}\Omega_{1,2}\Omega_{1,1}^{-1}$ needs to be estimated using a single observed time series. Now $\Omega_{1,1}$ can be estimated by replacing $\delta_0$ with the marginal likelihood estimates $\hat\delta$. However estimation of $\Omega_{1,2}$ is challenging, as $\Omega_{1,2} = n^{-1}E\left[\sum_{t=1}^n \sum_{s=1}^n \dot{l}_t(\delta_0)\dot{l}_s(\delta_0)\right]$ has cross terms $E\left(\dot{l}_t(\delta_0) \dot{l}_s(\delta_0)\right)$, $s\ne t$, which cannot be estimated without knowledge of $\psi_0$. We use the modified subsampling methods reviewed in \cite{wu2012variance} and \cite{wu2014parameter} to estimate $\Omega_{1,2}$.

Let $Y_{i,k_n}= (y_i,\ldots,y_{i+k_n-1})$ denote the subseries of length $k_n$ starting at the $i$th observation, where $i=1,\ldots, m_n$ and $m_n=n-k_n+1$ is the total number of subseries. Define
\[
\hat{q}_{n,t}= \frac{1}{\sqrt{n}}\dot{l}_{t}(\hat\delta)
\]
by replacing $\delta_0$ by $\hat\delta$ in \eqref{eq: deriv1 loglik t}. Under similar conditions to those given above, we show that as $k_n\to\infty$ and $m_n\to\infty$, $\hat\Gamma_{1,n}^{-1}\hat\Gamma^\dag_n\hat\Gamma_{1,n}^{-1}$ is a consistent estimator of the asymptotic covariance matrix of $\hat\delta$, where
\[
\hat\Gamma_{1,n} = \sum_{t=1}^n \hat{q}_{n,t}\hat{q}_{n,t}^{\T}; \quad
\hat\Gamma^\dag_n = \frac{1}{m_n}\sum_{i=1}^{m_n}\left(\sum_{t=i}^{i+k_n-1} \sum_{s=i}^{i+k_n-1}\hat{q}_{k_n,t}\hat{q}_{k_n,s}^{\T}\right)
\]
The performance of subsampling estimators relies on $k_{n}$ to a large extent. Following the guidance of \cite{heagerty2000window} on optimal selection of $k_{n}$, we use $k_{n}=C[n^{1/3}]$, $C=1,2,4,8$ in the simulations. The one dimensional integrals in $\hat{q}_{n,t}$ can be easily obtained using the \textbf{R} function \textsf{integrate}.

\section{Degeneration of Marginal Likelihood Estimates}\label{Sec: Pile up probability}%

Even when the identifiability conditions are satisfied, in finite samples the marginal likelihood can be maximised at $\hat\tau=0$, in which case $\hat\beta$ degenerates to the ordinary GLM estimate $\tilde\beta$. Simulation evidence of Section \ref{Sec: Simulations} suggests that the chance of this occurring, even for moderate to large sample sizes, is large (up to $50$\% for binary data but decreasing rapidly as the number of trials $m$ increases). In this section we will derive two approximations for this probability. In both approximations we conclude that $P(\hat\tau=0)$ will be high whenever the range of $x_{nt}^{\T}\beta$ is such that $\dot{b}(x_{nt}^{\T}\beta)\approx a_0 + a_1(x_{nt}^{\T}\beta)$, where $a_{0}, a_{1}$ are constants. When this linear approximation is accurate the covariance matrix for the marginal likelihood estimates is obtained from the inverse of a near singular matrix and results in $\textrm{var}(\hat \tau)$ being very large so that $P(\hat\tau=0)$ is close to $50$\%. When there is a nontrivial probability of $\hat\tau=0$, the distribution of $\hat\beta$ for finite samples is better approximated by a mixture of two multivariate distributions weighted by $P(\hat\tau=0)$ and $P(\hat\tau >0)$.

\subsection{Estimating the probability of $\hat\tau=0$}

One approximation for the probability of $\hat\tau=0$ can be obtained using the asymptotic normal distribution provided in Theorem \ref{Thm: Consist&Asym Marginal Like Estimator}. Define
$\kappa_{2}= P(\sqrt{n}(\hat\tau-\tau_0)\le -\sqrt{n}\tau_0)$, then in the limit,
\begin{equation}\label{eq: kappa2}%
\bar\kappa_{2}= \Phi( -\sqrt{n}\tau_0/\sigma_{\tau}(\delta_0)); \quad  \sigma_{\tau}^2(\delta_0) = (\Omega_{1,1}^{-1}\Omega_{1,2}\Omega_{1,1}^{-1})_{\tau\tau}
\end{equation}
where $\Phi(\cdot)$ is the standard normal distribution.

An alternative approximation to $P(\hat \tau=0)$ can be based on the score function evaluated at $\hat\tau=0$. Consider the scaled score function $S_{1,n}(\tilde\beta) = 2 n^{-1}\partial l_{1}(\beta, \tau)/\partial \tau|_{\beta=\tilde\beta,\tau=0}$, which, using integration by parts, is
\begin{equation}\label{eq: ST GLM}%
S_{1,n}(\tilde\beta)= \frac{1}{n}\sum_{t=1}^{n}\left[(y_{t}- m_t\dot{b}(x_{nt}^{\T} \tilde\beta))^{2} - m_t \ddot{b}(x_{nt}^{\T}\tilde\beta)\right].
\end{equation}
Now, $\hat\tau=0$ implies $S_{1,n}(\tilde\beta)\le 0$ but not the converse, hence $P(\hat\tau=0)$ is bounded above by $P(S_{1,n}(\tilde\beta)\le 0)$.

In order to derive a large sample approximation to this probability we show, in Section \ref{SSc: Asymptotic Theory GLM Estimation}, that the large sample distribution of $\sqrt{n}(S_{1,n}(\tilde\beta)-c_S)/\sigma_S$ is standard normal,
where $c_S= \underset{n\to \infty} \lim E(S_{1,n}(\beta^\prime))$ and $\sigma_{S}^2= \underset{n\to\infty}\lim \mathrm{Var}(\sqrt{n}S_{1,n}(\tilde\beta))$. Define $\kappa_{1} = P(S_{1,n}(\tilde\beta)\le 0)$, it can then be approximated with
\begin{equation}\label{eq: kappa 1}%
\bar\kappa_{1}= \Phi( -\sqrt{n}c_S/\sigma_S).
\end{equation}
The quantities $c_S$ and $\sigma_S$ can be expressed analytically for some regression specifications. In simulations, the limits are computed using numerical integration. For the binary case in particular, the ratio $c_S/\sigma_S$ can be quite small resulting in a large value for $\bar \kappa_1$. We compare how well $P(\hat \tau =0)$ is estimated by $\bar\kappa_{1}$ and $\bar\kappa_{2}$ via simulations in Section
\ref{Sec: Simulations}, and conclude that $\bar \kappa_{1}$ is slightly more accurate in the situation covered there.

\subsection{Asymptotic Theory for GLM Estimates and Marginal Score}
\label{SSc: Asymptotic Theory GLM Estimation}

To develop the asymptotic distribution of $S_{1,n}(\tilde\beta)$, the asymptotic normality of $\sqrt{n}(\tilde{\beta}-\beta')$ is required.

\begin{thm}[Asymptotic normality of GLM\ estimators]\label{Thm: GLM asymptotics}%

Under Conditions \ref{Cond: mt} to \ref{Cond: Mixing}, the estimates $\tilde\beta$ maximising the likelihood \eqref{eq: loglikglm} satisfies $\tilde\beta\overset{p}\to \beta'$, and $\sqrt{n}(\tilde\beta -\beta') \to \textrm{N}(0, \Omega_{1}^{-1} \Omega_{2}\Omega_{1}^{-1})$ as
$n \to \infty$, in which
\[
\Omega_{1} = \underset{n\rightarrow\infty}\lim \frac{1}{n}\sum_{t=1}^{n}m_{t}\ddot{b}
(x_{nt}^{T}\beta^{\prime})x_{nt}x_{nt}^{\T}
\]
\begin{align*}
\Omega_{2} =& \underset{n\rightarrow \infty}\lim \frac{1}{n}\sum_{t=1}^{n}\sum_{s=1}^n m_t m_s \left( \int (\dot{b}(x_{nt}^{\T}\beta_0 + \alpha_t) - \dot{b}(x_{nt}^{\T}\beta')) (\dot{b}(x_{ns}^{\T}\beta_0+ \alpha_s)- \dot{b}(x_{ns}^{\T}\beta')
)g(\alpha_t,\alpha_s;\tau_0,\psi_0) d\alpha\right) x_{nt}x_{ns}^{\T}\\
+ & \underset{n\to \infty}\lim \frac{1}{n}\sum_{t=1}^n m_t\left(\int \ddot{b}(x_{nt}^{\T}\beta_0 + \alpha_t)g(\alpha_t;\tau_0)d\alpha \right)x_{nt}x_{nt}^{\T}
\end{align*}
\end{thm}

The proof of this theorem is given in Appendix B. It relies on concavity of the GLM log likelihood with respect to $\beta$. Standard results of a functional limit theorem are used to establish the above result, in a similar way as that used in \cite{davis2000autocorrelation}, \cite{davis2009negative} and \cite{wu2014parameter}.

In order to use Theorem \ref{Thm: GLM asymptotics} for practical purposes, first $\beta'$ needs to be determined and then $\Omega_1$, $\Omega_2$. Estimation of $\beta'$ would require knowledge of $\tau$, and the estimation of $\Omega_2$ would require both $\tau$ and $\psi$, neither of which can be estimated using the GLM\ procedure. Hence the theorem is of theoretical value only.

Based on Theorem \ref{Thm: GLM asymptotics} we can now derive the large sample distribution of the score function of the marginal likelihood evaluated at $\tilde \delta= (\tilde  \beta,0)$.
Because all derivatives of $b(\cdot)$ are uniformly bounded and $\tilde\beta\overset{p}\to\beta'$, hence
\begin{equation*}
\sqrt{n}\left(S_{1,n}(\tilde\beta) - E(S_{1,n}(\beta')) \right) = \sqrt{n}\left( S_{1,n}(\beta^\prime) - E(S_{1,n}(\beta^\prime)) \right) - J_{S}^{\T}\sqrt{n}(\tilde\beta - \beta^\prime) + o_p(1).
\end{equation*}
Since $n^{-1/2}\sum_{t=1}^n (y_t-m_{t}\dot{b}(x_{nt}^T\tilde\beta))x_{nt} = 0$ by definition of $\tilde\beta$, using Taylor expansion
\begin{equation*}
\sqrt{n}(\tilde \beta-\beta')=\Omega_{1}^{-1}\frac{1}{\sqrt{n}}\sum_{t=1}^{n} e_{t,\beta'}x_{nt} + o_p(1), \quad e_{t,\beta'}= y_t-m_t\dot{b}(x_{nt}^{\T}\beta').
\end{equation*}
Then it follows that
\begin{equation}\label{eq: ScoreVec tau=0}%
\sqrt{n}\left(S_{1,n}(\tilde\beta))-E(S_{1,n}(\beta^\prime))\right) - \left(U_{1,n} - J_{S}^{\T}U_{2,n}\right) \overset{p}\to 0
\end{equation}
where
\begin{equation}\label{eq: Js}%
J_{S} = \underset{n\to\infty}\lim \frac{1}{n}\sum_{t=1}^n \left[2 m_t^2 (\pi^0(x_{nt}^{\T} \beta_0) - \dot{b}(x_{nt}^T\beta')) \ddot{b}(x_{nt}^T\beta^\prime)+ m_t b^{(3)}(x_{nt}^T \beta^\prime)\right] x_{nt}
\end{equation}
\begin{equation}\label{eq: U1 U2}%
U_{1,n}:= \sqrt{n}\left(S_{1,n}(\beta^\prime) - E(S_{1,n}(\beta^\prime)) \right)= \frac{1}{\sqrt{n}} \sum_{t=1}^{n} e_{t,\beta^\prime} ^2 - E e_{t,\beta'}^2; \quad
U_{2,n}:= \frac{1}{\sqrt{n}}\sum_{t=1}^{n} e_{t,\beta^\prime}c_{nt}
\end{equation}
note $\pi^0(x_{nt}^{\T}\beta_0) = \int \dot{b}(x_{nt}^{\T}\beta_0+\alpha_t)g(\alpha_t,\tau_0)d\alpha$ and $c_{nt} = \Omega_{1}^{-1}x_{nt}$, which is a non-random vector.

Then the CLT for $\sqrt{n}\left(S_{1,n}(\tilde\beta)- E(S_{1,n}(\beta'))\right)$ follows the CLT for the joint vector of $(U_{1,n}, U_{2,n})$. Note that both sequences of $\{U_{1,t}\}$ and $\{U_{2,t}\}$ are strongly mixing by Proposition 1 in \cite{blais2000limit}. Then the CLT for mixing process proposed in \cite{davidson1992central} can be applied to show that $(U_{1,n}, U_{2,n})$ is asymptotically normally distributed with mean zero and covariance matrix
\[
\begin{pmatrix}
V_S & \Omega_{1}^{-1}K_{S}\\
K_{S}^{\T}\Omega_{1}^{-1} &   \Omega_{1}^{-1}\Omega_{2}\Omega_{1}^{-1}
\end{pmatrix}
\]
where $\Omega_{1}, \Omega_{2}$ are given in Theorem \ref{Thm: GLM asymptotics}, and
\[
V_S:=\underset{n\to\infty}\lim \left\{ \sum_{h=0}^{(n-1)}\left(\frac{1}{n} \sum_{t=1}^{n}\textrm{Cov}(e_{t,\beta^\prime}^2, e_{t+h, \beta^\prime}^2)\right)
+ \sum_{h=1}^{(n-1)}\left(\frac{1}{n} \sum_{t=h+1}^{n}\textrm{Cov}(e_{t,\beta^\prime}^2,
e_{t-h, \beta^\prime}^2)\right)\right\}
\]
\[
K_{S}: = \underset{n\to\infty}\lim \left\{ \sum_{h=0}^{(n-1)}\left(\frac{1}{n}\sum_{t=1}^{n}
E (e_{t,\beta'} e_{t+h,\beta'}^{2}) x_{nt}\right) + \sum_{h=1}^{(n-1)}\left(\frac{1}{n} \sum_{t=h+1}^{n}E (e_{t,\beta'} e_{t-h,\beta'}^{2}) x_{nt}\right)\right\}
\]

\begin{thm} \label{Thm: score dist}
Under the assumptions of Theorem \ref{Thm: GLM asymptotics}, as $n\to \infty$,
$\sqrt{n}\left(S_{1,n}(\tilde\beta)- E(S_{1,n}(\beta'))\right)/\sigma_{S}\overset{d}\to N(0, 1)$.
\end{thm}

\subsection{An approximate mixture distribution for $\hat\beta$} \label{SSc: mixture}%

\begin{thm}[Mixture distribution under finite samples] \label{Thm: Prob mixdensity sigma=0}%
Assume $\tau_{0}>0$, under Conditions \ref{Cond: mt} to \ref{Cond: Ident and Const}, in finite samples, distribution of $\sqrt{n}(\hat\beta-\beta_{0})$ can be approximated with the mixture
\[
\kappa F_{1}(c,\delta_0) + (1-\kappa)F_{2}(c,\delta_0),\quad \kappa=P(\hat\tau=0)
\]
in which $F_{1}(c,\delta_0)$ is $r$-dimensional multivariate distribution obtained through
$\sqrt{n}(\hat\beta-\beta')$, which is a skew normal distribution $U_{2,n}|U_{1,n} + \sqrt{n} E(S_{1,n}(\beta')) - 2 J_{S}^{\T} U_{2,n}\le 0$, based on the joint normality of
$(U_{1,n}, U_{2,n})$ given in Theorem \ref{Thm: score dist}; $F_{2}(c,\delta_0)$ is a $r$-dimensional skew normal distribution $\sqrt{n}(\hat\beta-\beta_{0})|\hat\tau > 0$, based on the joint normality of $N(0,\Omega_{1,1}^{-1}\Omega_{1,2}\Omega_{1,1}^{-1})$ in Theorem
\ref{Thm: Consist&Asym Marginal Like Estimator}. Moreover, $\kappa\to0$ as $n\to\infty$. %
\end{thm}
Remarks
\begin{enumerate}
\item The skew normal distribution is defined in \cite{gupta2004multivariate}.

\item If $\tau_0=0$, $\beta'=\beta_0$ and the value $\kappa=0.5$ in the above mixture is similar to that in \citet[Theorem I]{moran1971maximum}; when $\tau_{0}=a/\sqrt{n}$, $a\ge0$, above results are parallel to those in \citet[Theorem IV]{moran1971maximum} and based on the same reasoning. However Moran's results are for independence observations whereas our results require the serial dependence to be accounted for in the asymptotic results.

\item While the mixture provides a better theoretical description of the asymptotic distribution for marginal likelihood estimates when $m$ is small, in practice, the mixture distribution cannot be estimated without knowing the true values of $\beta_0$, $\tau_0$ and $\psi_0$. In simulations, the covariance matrix for the joint distribution of $\hat\beta$ and $\hat\tau$ is approximated with $\Sigma(\delta_0)=n^{-1}\Omega_{1,1}^{-1}\Omega_{1,2}\Omega_{1,1}^{-1}$, and based on $F_2(c,\delta_0)$, we calculate
\begin{equation}\label{eq: beta cond}%
E(\hat\beta-\beta_0|\hat\tau>0) = \Sigma_{\beta\tau}(\delta_0)\Sigma^{-1}_{\tau\tau}(\delta_0) E(\hat\tau-\tau_0|\hat\tau >0)
\end{equation}
\begin{equation}\label{eq: Vbeta cond}%
\mathrm{Var}(\hat\beta|\hat\tau>0) = \Sigma_{\beta\beta}(\delta_0) - \Sigma_{\beta\tau}(\delta_0) \Sigma^{-1}_{\tau\tau}(\delta_0)\Sigma_{\tau\beta}(\delta_0) + \Sigma_{\beta\tau}(\delta_0) \Sigma^{-2}_{\tau\tau}(\delta_0)\Sigma_{\tau\beta}(\delta_0)\mathrm{Var}(\hat\tau-\tau_0|\hat\tau>0)
\end{equation}
\end{enumerate}

\section{Simulation Results} \label{Sec: Simulations}

In this section we summarize results of several simulation studies to illustrate the key theoretical results derived above as well as to indicate circumstances under which $P(\hat \tau=0)$ is large in which case the mixture distribution of Theorem \ref{Thm: Prob mixdensity sigma=0} would provide a more accurate description.

For all examples we consider the simple linear trend with latent process $W_{0,t} = \beta_1 + \beta_2 (t/n) + \alpha_t$, in which $\alpha_t$ is assumed to be: $\alpha_t=\phi\alpha_{t-1} + \epsilon_t$, $\epsilon_t\overset{i.i.d}\sim N(0,\sigma^2_{\epsilon})$ where $\sigma^2_{\epsilon}$ is chosen to maintain $\textrm{Var}(\alpha_t)=1$. In all cases the true values are $\beta_0 = (1,2)$ and $\tau_0=1$ and $\phi$ varies in the interval $(-1,1)$. While simple, this example provides substantial insights into the behaviour of the marginal likelihood estimates as well as into problems that can arise. The simplicity of this example also allows us to obtain analytical calculations of key quantities and to provide some heuristic explanations of the non-standard distribution results which can arise, particularly for binary time series.

In all simulations reported later, the number of replications was $10,000$. The marginal likelihood estimates were obtained using the \textbf{R} package \textsf{lme4}. The frequency with which $\hat \tau =0$ is not package dependent other than the occasionaly case -- this was checked using our own implementation based on adaptive Gaussian quadrature and by comparing the results with those from SAS PROC MIXED. The first simulation (Section \ref{Sec: Example 1 Sim}) focuses on binary responses and illustrates that the distribution of marginal likelihood estimates $\hat\delta$ for this kind of data converge towards a mixture as proposed in Theorem \ref{Thm: Prob mixdensity sigma=0}, in which the $P(\hat\tau=0)$ can be approximated using the result of Theorem \ref{Thm: score dist} to good accuracy. The second experiment (Section \ref{Sec: Example 2 Sim}) studies the finite sample performance of $\hat\delta$ for binomial cases and shows that $P(\hat\tau=0)$ vanishes as $m_{t}$ increases or as $n\to\infty$, thus the distribution of $\hat\delta$ is multivariate normal as developed in Theorem \ref{Thm: Consist&Asym Marginal Like Estimator}. Finally (Section
\ref{Sec: Example 3 sim}) the method for estimation the covariance matrix for $\hat\delta$ proposed in Section \ref{SSc: CovMat Est}, is evaluated.

In order to implement the simulations in Section \ref{Sec: Analytical Exp Key Asym Quantities} we first derive some theoretical expressions for key quantities used to define the large sample distributions of Theorems \ref{Thm: Consist&Asym Marginal Like Estimator}, \ref{Thm: GLM asymptotics} and \ref{Thm: score dist} as well as for the estimates $\hat \kappa_1$ and $\hat \kappa_2$ for $P(\hat \tau = 0)$.

\subsection{Analytical Expressions for Asymptotic Quantities} \label{Sec: Analytical Exp Key Asym Quantities}

Key quantities required for implementation and explanation of the simulation results to follow are: \newline (1). The limit point $\beta'$ for the GLM estimate $\tilde \beta$.
\newline (2). $c_S$ and $\sigma_S$ appearing in Theorem \ref{Thm: score dist} and used to obtain the approximation $\bar \kappa_1$ for $P(\hat \tau = 0)$.
\newline (3). The asymptotic variance of $\hat \tau$ in Theorem \ref{Thm: Consist&Asym Marginal Like Estimator} used to obtain the approximation $\bar \kappa_2$ for $P(\hat \tau = 0)$.
\newline (4). Various quantities defining the mixture distribution in Theorem \ref{Thm: Prob mixdensity sigma=0}.

Throughout, the derivations are given for the case of deterministic regressors specified as $x_{nt} = h(t/n)$ for a suitably defined vector function $h$ as in Condition \ref{Cond: Reg Trend Type}a and in which the first component is unity in order to include the intercept term. Also, in order to reduce notational clutter we will assume that $m_t\equiv m$ (the number of binomial trials at all time points is the same). The analytical expressions involve various integrals which are computed using numerical integration either with the \textbf{R}-package \textsf{integrate} or using grid evaluation for $u$ with mesh $0.0001$ over the interval $[0,1]$. Calculation of $\Omega_{1,2}$ in Theorem \ref{Thm: Consist&Asym Marginal Like Estimator} and $K_S$, $V_S$ and $\Omega_2$ in Theorem \ref{Thm: score dist} to obtain the variance $\sigma_{S}^2$ for the theoretical upper bound of $P(\hat\tau=0)$ require evaluation of two-dimensional integrals of the form
\begin{equation*}
\mathrm{Cov}(\dot{l}_{t}(\delta_0), \dot{l}_{t+h}(\delta_0)) = \sum_{y_t=0}^{m_{t}} \sum_{y_{t+h}=0}^{m_{t+h}}\pi^0(y_t,y_{t+h})\dot{l}_{t}(\delta_0)\dot{l}_{t+h}(\delta_0).
\end{equation*}
For these, the integral expression for $\pi^0(y_{t}, y_{t+h})$ is approximated using adaptive Gaussian quadratic (AGQ) with 9 nodes for each of the two dimensions.

\subsubsection{Limit Point of GLM Estimation}\label{SSSc: Lim GLM Est}%
By numerically solving the non-linear system \eqref{eq: betaprime equation for 2a} with Newton Raphson iteration the limiting value of the GLM estimates is $\beta'=(0.8206,1.7574)$.

\subsubsection{Quantities needed for $\bar \kappa_1$}
The analytical expression for the limiting expectation of the scaled score with respect to $\tau$ evaluated at the limiting point $\beta'$ is.
\begin{align}\label{eq: Expect Covg}%
c_{S}: =& \underset{n\to \infty} \lim E(S_{1,n}(\beta^\prime)) \nonumber\\
= & \underset{n\to\infty}\lim \frac{1}{n}\sum_{t=1}^n \left[ m(m-1) \int \left(\dot{b}(x_{nt}^{\T}\beta_{0}+\alpha_t) - \dot{b}(x_{nt}^{T}\beta^\prime) \right)^{2}g(\alpha_{t},\tau_{0})d\alpha_t +
m(\dot{b}(x_{nt}^{\T}\beta^\prime) - \pi^0(x_{nt}^{\T}\beta_0))(2\dot{b}(x_{nt}^{\T}\beta^\prime) -1) \right] \nonumber\\
= & m(m-1)\int_{0}^{1}\int \left(\dot{b}(h(u)^{\T}\beta_{0}+\alpha) - \dot{b}(h(u)^{\T}\beta^\prime) \right)^{2}g(\alpha,\tau_{0})d\alpha du\nonumber \\
+ & m\int_{0}^{1}(\dot{b}(h(u)^{\T}\beta^\prime) - \pi^0(h(u)^{\T}\beta_0)) (2\dot{b}(h(u)^{\T}\beta^\prime) -1)du \nonumber\\
= & m(m-1)c_{1}+ mc_{2}.
\end{align}
Note that $c_{1}$ in \eqref{eq: Expect Covg} is strictly positive but make no contribution for binary responses (when $m=1$) in which case $c_{2}$ is the only term contributing to $c_S$. We have observed that $c_2$ is non-negative, in the simulations but we do not have a general proof of that. In that case $c_S$ is also non-negative for all $m$. We have observed in the simulations that $c_1$ is substantially larger than $c_2$ and as a result $\bar \kappa_1$ is small for non-binary responses ($m>1$) but can be large for binary responses because $c_{2}\approx 0$.

Recall that $\sigma_{S}^2= \underset{n\to\infty}\lim \mathrm{Var}(S_{1,n}(\tilde\beta))$. For the case where the latent process is i.i.d. we have
\begin{align}\label{eq: Var Covg}%
\sigma_{S}^2= & \int_{0}^{1} m\pi^0(h(u)^{\T}\beta_0)(1-\pi^0(h(u)^{\T}\beta_0))\left[1+ 2(m-3)\pi^0(h(u)^{\T}\beta_0)(1-\pi^0(h(u)^{\T}\beta_0))\right] \nonumber \\
+ & 4m^3\pi^0(h(u)^{\T}\beta_0)(1-\pi^0(h(u)^{\T}\beta_0))(\pi^0(h(u)^{\T}\beta_0) - \dot{b}(h(u)^T\beta'))^2 \nonumber \\
+ & 4 m^2\pi^0(h(u)^{\T}\beta_0)(1-\pi^0(h(u)^{\T}\beta_0))(1-2\pi^0(h(u)^{\T}\beta_0)) (\pi^0(h(u)^{\T}\beta_0)-\dot{b}(h(u)^{\T}\beta'))du \nonumber \\
- & 2 J_S^T\Omega_{1}^{-1} K_S + J_S^T \Omega_1^{-1}\Omega_{2}\Omega_{1}^{-1} J_S
\end{align}
in which
\begin{align*}
K_S = & \int_{0}^{1} \left[ m\pi^0(h(u)^{\T}\beta_0)(1-\pi^0(h(u)^{\T}\beta_0)) (1-2\pi^0(h(u)^{\T}\beta_0)) \right. \\
+ & \left. 2 m^2\pi^0(h(u)^{\T}\beta_0)(1-\pi^0(h(u)^{\T}\beta_0))(\pi^0(h(u)^{\T}\beta_0) - \dot{b}(h(u)^{\T}\beta'))\right] h(u)du\\
J_S = & \int_{0}^{1} \left[ 2 m^2(\pi^0(h(u)^{\T}\beta_0) - \dot{b}(h(u)^{\T}\beta^\prime)) \ddot{b}(h(u)^{\T}\beta^\prime) + m b^{(3)}(h(u)^{\T}\beta^\prime)\right] h(u)du
\end{align*}
similar expression of $\sigma_{S}^2$ can be obtained for the dependent cases where $\psi \ne 0$. However, serial dependence does not appear to make much difference in the chance of $\hat\tau=0$ at least in the simulations of Section \ref{Sec: Example 2 Sim}.

For the simulation example of the binary response case ($m=1$) the expressions \eqref{eq: Expect Covg}, \eqref{eq: Var Covg} can be evaluated using numerical integration to give
\begin{table}[h]\centering
	\begin{tabular}{l*{4}{c}r}
		\multicolumn{4}{l}{$c_{1}=0.0168$, $c_{2}=1.65\times 10^{-5}$}\\ \hline
		& $c_{S}$ & $\sigma_S$ & $c_{S}/\sigma_{S}$ \\ \hline
		$m=1$ & $1.65\times 10^{-5}$ & $7.17\times 10^{-3}$ & 0.0023 \\
		$m=2$ & 0.034 & 0.303 & 0.1110 \\
		$m=3$ & 0.101& 0.525 & 0.1921\\ \hline
	\end{tabular}
\end{table}

Since $c_{S}$ is observed to be strictly positive for marginal likelihood, $\sqrt{n}c_{S}\to \infty$ as $n\to \infty$. By Theorem \ref{Thm: score dist}, the probability of $S_{1,n}(\tilde\beta)\le 0$ vanishes as $n\to \infty$. Because $P(\hat\tau=0)$ is bounded above by $P(S_{1,n}(\tilde\beta)\le0)$, as $n\to \infty$, the marginal likelihood estimate will be such that $\hat\tau=0$ with vanishing probability. However, notice for binary data $c_S=c_2$ and $c_S/\sigma_S = 0.0023$. Hence, even for the largest sample size ($n=5000$) reported below, $\bar \kappa_1$ is $44\%$. It would require a sample size of $10^6$ to reduce this to $1\%$. Clearly this has substantial implications for the use of marginal likelihood for binary data. For binomial responses the $c_1$ term dominates, and hence even with small values of $m>1$ the chance of $\hat \tau =0$ reduces rapidly.

We conclude this subsection with a heuristic explanation of why $c_2 \approx 0$. If $h(u)^\T \beta'$ appearing in the definition of $c_2$ is such that $\dot b (h(u)^\T \beta')$ is well approximated linearly in $h(u)^\T \beta'$ then $c_2$ could be approximated by
\begin{equation}
c_2^*=\int_0^1 [\dot b(h(u)^{\T} \beta')-\pi^0(h(u)^\T \beta_0)](a_o+a_1 h(u)^\T \beta')du.
\end{equation}
But, because the first element of $h(u)$ is equal to $1$, $a_o+a_1 h(u)^\T \beta'$ can be rewritten as $h(u)^\T \beta^{*}$ for some vector $\beta^*$. Then $c_2^*$ can be written as
\begin{equation}
c_2^*=\int_0^1 [\dot b(h(u)^\T\beta')-\pi^0(h(u)^\T \beta_0)]h(u)^\T du\beta^* = 0
\end{equation}
by definition of $\beta'$ in \eqref{eq: betaprime equation for 2a}. Note that $\dot b(x)$ is the probability of success from a logit response and hence if $x$ ranges over reasonably large values then $\dot b (x)$ may be near linear. In the example used for simulations $h(u) \beta'$ ranges over the interval $(0.8206, 2.578)$ and $\dot b (x)$ for $x$ in this interval is well approximated by a straight line in $x$.

\subsubsection{Asymptotic Covariance matrix for marginal estimates} \label{SSc: asym CovMat}%

Although positive definite, $\Omega_{1,1}$ of the asymptotic covariance matrix in Theorem
\ref{Thm: Consist&Asym Marginal Like Estimator} can be near singular and this results in an overall covariance matrix for $\sqrt{n}(\hat \delta - \delta_0)$ which has very large elements; in particular, the variance of $\hat \delta$ in \eqref{eq: kappa2} is very large and, as a result $\bar \kappa_2$ given in \eqref{eq: kappa2} will also be close to $50\%$. The reason for this will
be analyzed in this section by calculating various components of the asympototic covariance for marginal estimates. To keep the discussion manageable the deterministic regressors, assumed to be generated by functions $h(\cdot)$ and to satisfy Condition \ref{Cond: Reg Trend Type}a, will be used for the derivations. In this case, the summations on the left of \eqref{eq: PD InfMat} has limit given by the integral
\begin{equation}\label{eq: PD InfMat Lim}%
\Omega_{1,1} =  \int_{0}^{1} \sum_{y=0}^{m}f(y|h(u),\delta_0)\dot{l}(y|h(u);\delta_0) \dot{l}^{\T}(y|h(u);\delta_0)du
\end{equation}
in which
\begin{equation*}
\dot{l}(y|h(u);\delta) = f^{-1}(y|h(u),\delta)\begin{pmatrix}\int (y - m \dot{b}(h(u)^{\T}\beta + \sqrt{\tau}z))f(y|h(u),z,\delta)\phi(z)dz \cdot h(u)\\
\int \left[(y- m \dot{b}(h(u)^{\T}\beta +\sqrt{\tau}z))^2-m\ddot{b}(h(u)^{\T}\beta + \sqrt{\tau}z) \right] f(y|h(u),z,\delta)\phi(z)dz /2
\end{pmatrix}
\end{equation*}
where the first dimension corresponds to $\partial l(y|h(u);\delta)/\partial \beta$ and the second dimension corresponds to $\partial l(y|h(u);\delta)/\partial \tau$.

For binary responses, using the facts that $m\equiv 1$ and $y^2=y$, the component in $\partial l(y|h(u);\delta)/\partial \tau$ has
\begin{equation}\label{eq: Trans tau binary}%
(y- m \dot{b}(h(u)^{\T}\beta + \sqrt{\tau}z))^2-m\ddot{b}(h(u)^{\T}\beta + \sqrt{\tau}z) =
(y-\dot{b}(h(u)^{\T}\beta+\sqrt{\tau}z))(1-2\dot{b}(h(u)^{\T}\beta+ \sqrt{\tau}z)).
\end{equation}
Let $W=h(u)^{\T}\beta+\sqrt{\tau}z$, note for binary cases, $(y-\dot{b}(W))f(y|h(u),z,\delta) =
\ddot{b}(W)$ if $y=1$ and $-\ddot{b}(W)$ if $y=0$. Define conditional distribution $\rho(z) = \ddot{b}(W)\phi(z)/\int \ddot{b}(W) \phi(z)dz$, then if $h(u)^{\T}\beta$ appearing in $\dot{l}(y|h(u);\delta)$ is such that the linearly approximation
\begin{equation}\label{eq: d1 tau cond}%
\frac{\partial l(y|h(u);\delta)/\partial \tau}{\int \ddot{b}(W)\phi(z)dz } = \int (1-2\dot{b}(W))\rho(z)dz\approx a_0^\ast + a_1^\ast (h(u)^{\T}\beta) = h(u)^{\T}\beta^\ast
\end{equation}
is well established, then for each fixed $u\in [0,1]$ and $\delta$, $\dot{l}(y|h(u);\delta)$ is a nearly linear dependent vector. Consequently $\Omega_{1,1}$ becomes near singular with a large inverse. In this example $h(u)^{\T}\beta_0$ takes all values in the interval $(1,3)$ over which the left side of
\eqref{eq: d1 tau cond} is approximately a straight line and hence, in view of the above discussion,  and $\Omega_{1,1}$ has eigenvalues $(0.129,6.784\times 10^{-3}, 1.903\times10^{-6})$, with $\sigma_{\tau}^2(\delta_0)=698$.

For binomial data, the properties of $m > 1$ and $y^2 \ne y$ and the left side of \eqref{eq: Trans tau binary} is no longer a near linear function of $h(u)$. As a result $\Omega_{1,1}$ is not nearly singular and $\sigma_{\tau}^2(\delta_0)$ is of moderate size.

\subsubsection{Quantities needed for mixture distribution}
To assess the accuracy of the asymptotic mixture distribution, the theoretical mean vector and covriance for the distributions $F_1(\cdot,\delta_0)$ and $F_2(\cdot,\delta_0)$ are required. For $F_1(\cdot,\delta_0)$ these are approximately as for the normal distribution in Theorem \ref{Thm: GLM asymptotics} for the GLM estimates. The mean is $\beta' = (0.82, 1.76)$ given in Section \ref{SSSc: Lim GLM Est}. The covariance matrix
\begin{equation*}
\Omega_{1}^{-1}\Omega_{2}\Omega_{1}^{-1} = \begin{pmatrix}
23.636 & -39.8212\\
-39.8212 & 98.13 \end{pmatrix}
\end{equation*}
is calculated using $\delta_0=1$ and $\beta'$, in
\begin{equation*}
\Omega_{1}=\int_{0}^{1} \ddot{b}(h^{T}(u)\beta')h(u)h^{T}(u)du
\end{equation*}
and
\begin{equation*}%
\Omega_{2} = \int_{0}^{1} \left[\pi^0(h(u)^{\T}\beta_0)- 2\pi^0(h(u)^{\T}\beta_0) \dot{b}(h(u)^{\T}\beta') + \dot{b}^2(h(u)^{\T}\beta')\right]h(u)h(u)^{\T}du.
\end{equation*}

For $F_{2}(\cdot,\delta_0)$, the mean is evaluated with conditinoal mean \eqref{eq: beta cond} and covariance matrix \eqref{eq: Vbeta cond}. In this example $\Sigma(\delta_0)= n^{-1}\Omega_{1,1}^{-1}$, where $\Omega_{1,1}$ is provided above. For binary data, $\mathrm{E}(\hat\tau-\tau_0|\hat\tau>0)$ and $\mathrm{Var}(\hat\tau-\tau_0|\hat\tau>0)$ are obtained empirically using $\hat\Omega_{1,1}$ as follows:
\begin{align*}
\hat\Omega_{1,1}= & \frac{1}{n}\sum_{t=1}^n\left\{ \dot{l}_t(\delta^\ast)\dot{l}_t^T(\delta^\ast) + f(y_t|x_{nt};\delta^\ast)^{-1}\int (y_t-m_t\dot{b}(W_t^\ast))
\begin{pmatrix} 0 & 0\\
0 & \frac{1}{4\tau^\ast}
\end{pmatrix}
f(y_t|x_{nt},z_t,\delta^\ast)\phi(z)dz \right. \\
- & \left. f(y_t|x_{nt};\delta^\ast)^{-1}\int \left[ (y_t-m_t\dot{b}(W_t^\ast))^2 - m_t\ddot{b}(W_t^\ast)\right]\binom{x_{nt}}{\frac{z_t}{2\sqrt{\tau^\ast}}}
(x_{nt}^T,\frac{z_t}{2\sqrt{\tau^\ast}})f(y_t|x_{nt},z_t,\delta^\ast)\phi(z)dz \right\}
\end{align*}
where $\Vert \delta^\ast -\hat\delta\Vert\le \Vert\hat\delta-\delta_0\Vert$ but under finite samples $\delta^\ast\ne \delta_0$. Note that although for binary data, $\dot{l}_{t}(\delta)$ can be almost linearly dependent as analyzed above, the last part in $\hat\Omega_{1,1}$ is nontrivial and nonsingular for $\delta^\ast\ne \delta_{0}$, which allows $\hat\Omega_{1,1}$ to be nonsingular. As a result a large difference between theoretical and empirical covariance is observed for binary data, which will be shown in Example 2 of simulations.

\subsection{Example 1: binary data, independent latent process} \label{Sec: Example 1 Sim}

This example considers the simplest case of independent observations obtained when $\tau_0=1$ and $\phi_0=0$. For each replication, an independent binary sequence of $Y_t|\alpha_t,x_{nt}\sim B(1,p_t)$, $p_t=1/(1+\exp(-W_{0,t}))$, is generated. Table \ref{tb: conv mixture coef} reports the empirical values $\hat\kappa_{1}$ (the empirical proportion of $S_{1,n}(\tilde\beta)\le 10^{-6}$) and  $\hat\kappa_{2}$ (the empirical proportion of $\hat\tau\le 10^{-6}$). Also shown are the empirical mean and standard deviation (in parentheses) of $\hat\beta$ conditional on $\hat\tau=0$ and $\hat\tau>0$ obtained from the simulations along with the theoretical values of these obtained from \eqref{eq: kappa 1}, \eqref{eq: kappa2}, and \eqref{eq: beta cond}, \eqref{eq: Vbeta cond} associated with Theorem \ref{Thm: Prob mixdensity sigma=0}.

Table \ref{tb: conv mixture coef} clearly demonstrates that for this data generating mechanism there is very high proportion of replicates for which $\hat \tau = 0$ and that this proportion does not decrease rapidly with sample size increasing. This is as predicted by theory and the theoretical values of $\bar \kappa_1$ and $\bar \kappa_2$ provide good approximations to $\hat \kappa_1$ and $\hat \kappa_2$. As explained above, this high proportion of zero estimates for $\tau$, even for large sample sizes, is as expected for the regression structure used in this simulation. Note that $\bar\kappa_{1}\le \bar\kappa_{2}$ and $\bar\kappa_{1}$ is closer to the probability of $\hat\tau=0$. The estimations reverse the theoretical property that $\kappa_{1}\ge \kappa_{2}$. For binary data, $P(\hat\tau=0)$ and both theoretical approximations, $\bar{\kappa}_1$ and $\bar{\kappa}_2$ decrease slowly and a very large sample is required to attain $P(\hat\tau=0)\approx 0$.

It is also clear from Table \ref{tb: conv mixture coef} that the empirical mean and standard deviation of $\hat\beta|\hat\tau=0$ and $\hat\beta|\hat\tau>0$ show good agreement with the theoretical results predicted by Theorem \ref{Thm: Prob mixdensity sigma=0}. Overall, the theory we derive for $P(\hat\tau=0)$ and the use of a mixture distribution for $\hat\beta$ is quite accurate for all sample sizes, and with relatively large sample the mixture is a better representation. However, estimation of the corresponding distributions in the mixture requires $\beta_0$, $\tau_0$ and $\psi_0$ and therefore cannot be implemented in practice.
\begin{table}[ptb]\centering
\begin{tabular}{l|llll|l*{4}{c}r}\hline
& \multicolumn{4}{c}{Theoretical} & \multicolumn{5}{c}{Empirical from simulations}\\
& $\hat\beta|\hat\tau=0$ & $\hat\beta|\hat\tau>0$ & $\bar\kappa_1$ & $\bar\kappa_2$ & $\hat\beta|\hat\tau=0$ & $\hat\beta|\hat\tau>0$ & $\hat\kappa_{1}$ & $\hat\kappa_{2}$ &\\ \hline
\multirow{2}{*}{$n=200$} & 0.82(0.344)&1.10(0.421)& \multirow{2}{*}{48.70\%}& \multirow{2}{*}{49.20\%} & 0.83(0.354)& 1.07(0.431) & \multirow{2}{*}{45.56\%}&\multirow{2}{*}{45.54\%}\\
& 1.76(0.700)&2.15(0.819)& & &1.79(0.731) & 2.26(1.020) \\ \hline
\multirow{2}{*}{$n=500$}&0.82(0.217)&1.05(0.269)& \multirow{2}{*}{47.95\%} & \multirow{2}{*}{48.72\%} & 0.82(0.219) &1.04(0.261) &\multirow{2}{*}{46.27\%} &\multirow{2}{*}{46.26\%}\\
&1.76(0.443)&2.07(0.521) & & & 1.78(0.451)& 2.10(0.621)\\  \hline
\multirow{2}{*}{$n=10^3$}&0.82(0.154)& 1.02(0.193)& \multirow{2}{*}{47.10\%} & \multirow{2}{*}{48.20\%}& 0.82(0.156)& 1.01(0.182) &\multirow{2}{*}{45.52\%}
&\multirow{2}{*}{45.51\%} \\
&1.76(0.313) & 2.03(0.371)& & & 1.77(0.315)& 2.06(0.415)\\ \hline
\multirow{2}{*}{$n=5\cdot10^3$}& 0.82(0.069)&0.97(0.093)& \multirow{2}{*}{43.54\%} & \multirow{2}{*}{45.96\%}& 0.82(0.069) &0.96(0.089) &\multirow{2}{*}{42.93\%}
&\multirow{2}{*}{42.93\%}\\
& 1.76(0.140)&1.95(0.174)& & &1.76(0.139)& 1.96(0.184) \\ \hline
\multicolumn{4}{l}{\small Standard deviation in ``()"}
\end{tabular}
\caption{Mixture distribution of marginal likelihood estimates for binary independent time series} \label{tb: conv mixture coef}%
\end{table}

\subsection{Example 2: binomial data, correlated latent process}
\label{Sec: Example 2 Sim}

This simulation investigates bias and standard deviation of the marginal estimates for $m= 1, 2, 3$, $n = 200, 500$ and a range of serial dependence given by $\phi = -0.8, -0.2, 0, 0.2, 0.8$. The observed standard deviation of the estimates over the replications is compared to that given by the asymptotic covariance matrix in Theorem \ref{Thm: Consist&Asym Marginal Like Estimator}. We also give the empirical proportion of the event that $\hat \tau = 0$ using the proportion $\hat\tau\le10^{-6}$. The theoretical upper bound $\bar\kappa_{1}$ is approximated with $\bar{\kappa}_1$ defined in \eqref{eq: kappa 1}.

Table \ref{tb: asym MGLMM binomial} summarizes the results. For binomial series ($m\ge 2$), the empirical values are in good agreement with the asymptotic mean and standard deviation with bias of the estimates generally improving with $m$ or $n$ increasing. Also observed for binomial series is that, both theoretically and empirically, the probability of $\hat\tau=0$ decreases quickly by increasing either the number of trials $m$ or the sample size $n$. However, for binary responses, large asymptotic standard deviations are obtained for the reason explained in Section
\ref{SSc: asym CovMat}. In binary data, the probability of $\hat\tau=0$ is close to $50\%$ and although theoretically this ($\bar\kappa_1$) will converges to zero as $n\to\infty$, it is doing so slowly as $n$ increases. This can be explained by Theorem \ref{Thm: score dist}. Under the settings of this example, using AGQ,  $\sigma_{S}$ is calculated to vary from $0.72\times10^{-2}$ to $1.55\times10^{-2}$ and the values of $c_{1}$ and $c_{2}$ are the same as those in Example 1 and hence for, binary responses, $\sqrt{n}c_{S}/\sigma_S$  is at most $2\sqrt{n}\times 10^{-3}$ across the range of autocorrelation considered here. Thus the large values of  $\bar\kappa_{1}=\Phi(-\sqrt{n}c_{S}/\sigma_{S})$ across the range of autocorrelations is to be expected for this regression structure. However, for binomial series with $m=2$ and $n=200$, $\phi=0.8$ for instance, $\sigma_S\approx 0.351$, $\sqrt{n}c_{S}$ is dominated by $\sqrt{n}m(m-1)c_{1}=0.475$ which explains why $P(\hat\tau=0)$ decreases rapidly with $m>1$ and increasing $n$.

Interestingly, for binary series, as $n$ increases from 200 to 500 the bias of $\hat\tau$ worsens which seems somewhat counterintuitive. A plausible explanation for this is that the distribution of $\hat\tau$ is a mixture using weights
$P(\hat\tau=0)$, which is approximately $45\%$ across sample size and serial dependence, and $P(\hat\tau>0)$.
When $n=200$ the conditional distribution of $\hat\tau|\hat\tau>0$ has larger variance than when $n=500$ resulting in an inflated overall mean when $n=200$ relative to $n=500$.

In summary, the theoretical upper bound ($\bar\kappa_1$) is above or close to the empirical proportion of $P(\hat\tau=0)$, which is a pattern that is consistent with Theorem
\ref{Thm: score dist}. For binomial series ($m \ge 2$) the the marginal estimates have good bias properties and standard deviations explained by the large sample distribution of Theorem \ref{Thm: Consist&Asym Marginal Like Estimator} and these conclusions are not severely impacted by the level or direction of serial dependence. For binary series, the high proportion of $\hat \tau= 0$ is persistent regardless of $n=200, 500$ or the level of serial dependence and this is explained by theory presented above.

\begin{table}[ptb]\small\centering
\begin{tabular}[c]{l*{12}{c}r}
\multicolumn{10}{l}{$n=200$} \\ \hline
&&\multicolumn{3}{c}{$m=1$} & \multicolumn{3}{c}{$m=2$} & \multicolumn{3}{c}{$m=3$}\\
$\phi$ & & Mean & SD & ASD & Mean & SD & ASD & Mean & SD &ASD \\ \hline
\multirow{4}{*}{0.8} & $\hat\beta_1$&0.984&0.617&10.22&0.994&0.525&0.516&0.991&0.484&0.483\\
& $\hat\beta_2$&2.102&1.207&15.19&2.046&0.975&0.942&2.023&0.890&0.878\\
&$\hat\tau$&0.983&1.177&65.20&0.960&0.790&0.821&0.909&0.541&0.557\\
&$\hat\kappa(\bar\kappa_1)$ & \multicolumn{3}{c}{45.41\%(49.44\%)} &\multicolumn{3}{c}{11.06\%(8.85\%)} & \multicolumn{3}{c}{2.52\%(3.00\%)}\\ \hline
\multirow{4}{*}{0.2} & $\hat\beta_1$ &0.953&0.433&7.90&0.996&0.346&0.342&0.992&0.293&0.292\\
&$\hat\beta_2$&2.071&0.949&11.84&2.047&0.667&0.648&2.027&0.563&0.553\\
&$\hat\tau$&0.898&1.026&50.55&1.052&0.793&0.790&0.993&0.518&0.510\\
&$\hat\kappa(\bar\kappa_1)$& \multicolumn{3}{c}{45.08\%(49.38\%)}&
\multicolumn{3}{c}{8.08\%(7.15\%)} & \multicolumn{3}{c}{0.86\%(1.30\%)}\\ \hline
\multirow{4}{*}{0} & $\hat\beta_1$ &0.955&0.414&7.71&1.007&0.331&0.327&1.001&0.278&0.273 \\
& $\hat\beta_2$ &2.048&0.926&11.56&2.032&0.634&0.623&2.013&0.527&0.523\\
& $\hat\tau$ &0.869&1.015&49.35&1.063&0.786&0.789&1.003&0.512&0.509 \\
&$\hat\kappa(\bar\kappa_1)$& \multicolumn{3}{c}{46.72\%(49.37\%)} & \multicolumn{3}{c}{7.25\%(7.08\%)} & \multicolumn{3}{c}{0.78\%(1.24\%)}\\ \hline
\multirow{4}{*}{-0.2}& $\hat\beta_1$ &0.958&0.408&7.58&1.003&0.317&0.316&0.999&0.261&0.261\\
& $\hat\beta_2$&2.049&0.913&11.39&2.036&0.620&0.606&2.014&0.511&0.504\\
& $\hat\tau$ &0.880&1.017&48.57&1.059&0.794&0.789&1.004&0.516&0.510\\
&$\hat\kappa(\bar\kappa_1)$& \multicolumn{3}{c}{45.78\%(49.37\%)}  & \multicolumn{3}{c}{7.27\%(7.07\%)} &\multicolumn{3}{c}{0.8\%(1.23\%)} \\ \hline
\multirow{4}{*}{-0.8} &$\hat\beta_1$ &0.953&0.405&7.53&0.995&0.306&0.303&0.992&0.247&0.245\\
& $\hat\beta_2$ &2.045&0.932&11.34&2.038&0.623&0.603&2.023&0.512&0.500\\
& $\hat\tau$ &0.863&1.006&48.30&1.047&0.794&0.812&1.002&0.547&0.542\\
& $\hat\kappa(\bar\kappa_1)$ & \multicolumn{3}{c}{46.01\%(49.40\%)}& \multicolumn{3}{c}{7.92\%(7.88\%)} & \multicolumn{3}{c}{1.43\%(1.91\%)}  \\ \hline
&\\
\multicolumn{10}{l}{$n=500$} \\ \hline
&& \multicolumn{3}{c}{$m=1$} & \multicolumn{3}{c}{$m=2$} &\multicolumn{3}{c}{$m=3$} \\
$\phi$ & & Mean & SD & ASD & Mean & SD & ASD & Mean & SD &ASD \\ \hline
\multirow{4}{*}{0.8} & $\hat\beta_1$ &0.948&0.373&6.58&0.998&0.335&0.330&0.998&0.311&0.309\\
&$\hat\beta_2$&1.992&0.737&9.77&2.013&0.615&0.604&2.001&0.573&0.563\\
&$\hat\tau$&0.785&0.866&41.96&0.974&0.515&0.519&0.956&0.35&0.352&\\
&$\hat\kappa(\bar\kappa_1)$ & \multicolumn{3}{c}{44.77\%(49.36\%)}& \multicolumn{3}{c}{1.45\%(1.68\%)} & \multicolumn{3}{c}{0.07\%(0.14\%)}\\ \hline
\multirow{4}{*}{0.2} & $\beta_1$ &0.936&0.273&5.00&1.000&0.217&0.216&1.002&0.182&0.184\\
&$\hat\beta_2$ &1.953&0.589&7.50&2.012&0.412&0.410&2.002&0.349&0.350\\
&$\hat\tau$ & 0.694&0.772&31.99&1.011&0.492&0.499&0.998&0.325&0.322\\
&$\hat\kappa(\bar\kappa_1)$ & \multicolumn{3}{c}{46.18\%(49.32\%)} & \multicolumn{3}{c}{0.93\%(1.06\%)} & \multicolumn{3}{c}{0\%(0.02\%)}\\ \hline
\multirow{4}{*}{0} &$\hat\beta_1$ &0.938&0.263&4.87&0.997&0.205&0.206&0.999&0.173&0.172\\
&$\hat\beta_2$&1.959&0.572&7.32&2.011&0.395&0.394&2.005&0.333&0.331\\
&$\hat\tau$&0.708&0.767&31.21&1.011&0.502&0.499&0.996&0.321&0.322\\
&$\hat\kappa(\bar\kappa_1)$& \multicolumn{3}{c}{45.18\%(49.32\%)} & \multicolumn{3}{c}{1.04\%(1.04\%)} & \multicolumn{3}{c}{0\%(0.017\%)}\\ \hline
\multirow{4}{*}{-0.2} &$\hat\beta_1$ &0.934&0.257&4.79&1.000&0.202&0.200&0.996&0.164&0.164\\
&$\hat\beta_2$ &1.957&0.555&7.21&2.013&0.389&0.383&2.010&0.315&0.318\\
&$\hat\tau$ &0.689&0.763&30.71&1.014&0.493&0.499&0.998&0.329&0.322\\
&$\hat\kappa(\bar\kappa_1)$&\multicolumn{3}{c}{46.26\%(49.32\%)} &
\multicolumn{3}{c}{0.7\%(1.04\%)} &\multicolumn{3}{c}{0.01\%(0.017\%)}\\ \hline
\multirow{4}{*}{-0.8} &$\hat\beta_1$ &0.929&0.253&4.76&0.996&0.190&0.191&0.996&0.156&0.155\\
& $\hat\beta_2$ &1.962&0.570&7.18&2.012&0.386&0.381&2.011&0.319&0.317\\
& $\hat\tau$ &0.681&0.768&30.57&1.008&0.509&0.513&0.997&0.344&0.343\\
&$\hat\kappa(\bar\kappa_1)$& \multicolumn{3}{c}{47.10\%(49.34\%)} & \multicolumn{3}{c}{1.01\%(1.31\%)} & \multicolumn{3}{c}{0.02\%(0.05\%)}\\ \hline
\end{tabular}
\caption{Marginal likelihood estimates for Binomial observations under various values of $\phi$, where the true values are $\beta^0_1=1$; $\beta^0_2=2$; $\tau_0=1$.} \label{tb: asym MGLMM binomial}%
\end{table}

\subsection{Example 3: Estimate of Covariance Matrix}\label{Sec: Example 3 sim}

The subsampling method of estimating the covariance for marginal likelihood estimates is of limited practical value for binary data because, firstly, when $\hat\tau=0$, which occurs nearly 50\% of the time, the method does not provide estimates of the covariance matrix in Theorem \ref{Thm: Consist&Asym Marginal Like Estimator} and, secondly, because of the high proportion of $\hat\tau=0$, $\hat\delta$ has the mixture distribution given in Theorem \ref{Thm: Prob mixdensity sigma=0}, the covariance of which requires $\beta'$ and $\delta_0$, both of which are unknown and cannot be estimated from a single sequence.

In the binomial case ($m \ge 2$) the subsampling method is likely to be useful for a range of serial dependence.
Table \ref{tb: cov binom subsampling 2} presents simulation results under various levels of serial correlation. The table summarizes the estimates of standard deviation for $\hat\delta$ using the subsampling method described in Section \ref{SSc: CovMat Est}. The column ``$\textrm{ASD}$" contains the asymptotic standard deviation calculated with the covariance matrix in Theorem \ref{Thm: Consist&Asym Marginal Like Estimator} and column ``$\textrm{SD}$" contains empirical standard deviation. The table shows the subsampling estimates of standard deviations are of the same magnitude of theoretical standard deviations even for moderate sample size $n=200$ but are biased downwards and increasingly so as $C$ increases. Values of $C=1,2$ provide the least biased estimates for the standard errors for both sample sizes. Downwards bias is greater for large positive values of $\phi$ as might be expected.
\begin{table}[ptb]\centering
\begin{tabular}[c]{l*{8}{c}r}
\multicolumn{2}{l}{\small $m=2$, $n=200$} \\\hline
&& \textrm{ASD} & \textrm{SD} & $k_{n}=5$ & $k_{n}=11$ & $k_{n}=23$ &$k_{n}=46$ \\
\multirow{3}{*}{$\phi=0.8$} & $\hat\beta_1$ & 0.515 & 0.525 & 0.392 & 0.406 & 0.387 &0.325\\
& $\hat\beta_2$ & 0.942 & 0.975 & 0.734 & 0.755& 0.719 & 0.606\\
& $\hat\tau$ & 0.821 & 0.790 & 0.824& 0.811& 0.784 & 0.733\\ \hline
\multirow{3}{*}{$\phi=0.2$} & $\hat\beta_1$ & 0.342 & 0.346 & 0.333&0.316 &
0.284&0.235\\
& $\hat\beta_2$ & 0.648 & 0.667& 0.637 & 0.605 & 0.545& 0.447\\
& $\hat\tau$ & 0.790 & 0.793 &0.829& 0.814 & 0.786& 0.729\\ \hline
\multirow{3}{*}{$\phi=-0.2$} & $\hat\beta_1$ & 0.316 & 0.322&0.313&0.296&0.265 &0.221\\
& $\hat\beta_2$ & 0.606 & 0.616&0.604& 0.570 &0.511& 0.420\\
& $\hat\tau$ & 0.789 & 0.785& 0.831& 0.816& 0.786& 0.729\\ \hline
\multirow{3}{*}{$\phi=-0.8$} &$\hat\beta_1$& 0.303 & 0.305&0.300&0.283&0.255&0.214\\
&$\hat\beta_2$ & 0.603 & 0.619& 0.592 & 0.562& 0.508& 0.419\\
&$\hat\tau$& 0.812 & 0.796 & 0.842& 0.834 & 0.808& 0.750\\ \hline
\\
\multicolumn{2}{l}{\small $m=2$, $n=500$} \\\hline
&& \textrm{ASD} & \textrm{SD} & $k_{n}=7$ & $k_{n}=14$ & $k_{n}=28$ &$k_{n}=56$ \\
\multirow{3}{*}{$\phi=0.8$} & $\hat\beta_1$ & 0.330&0.330& 0.261 & 0.275 & 0.271 & 0.242\\
&$\hat\beta_2$ & 0.604 &0.607& 0.487 & 0.510& 0.501& 0.451\\
&$\hat\tau$&0.519 & 0.507&0.506 &0.501 &0.490 &0.469\\ \hline
\multirow{3}{*}{$\phi=0.2$} & $\hat\beta_1$ &0.216 & 0.217 &0.210&0.204 &0.192 & 0.169\\
&$\hat\beta_2$ & 0.410 &0.418 & 0.400&0.388 & 0.364& 0.321\\
&$\hat\tau$& 0.499 &0.493 & 0.503&0.497&0.486&0.465\\ \hline
\multirow{3}{*}{$\phi=-0.2$} & $\hat\beta_1$ & 0.199 & 0.199 & 0.197& 0.190 &0.179 &0.159\\
& $\hat\beta_2$ &0.383 & 0.387 & 0.378& 0.365 &0.343&0.303 \\
& $\hat\tau$ &0.499 & 0.503&0.505& 0.500& 0.490&0.469\\ \hline
\multirow{3}{*}{$\phi=-0.8$} &$\hat\beta_1$& 0.191 & 0.191 & 0.188 & 0.182& 0.171& 0.152\\
&$\hat\beta_2$& 0.381 & 0.387 & 0.372& 0.361 &0.340 & 0.302\\
&$\hat\tau$&0.513&0.513&0.514 &0.511&0.502&0.481\\ \hline
\end{tabular}
\caption{Subsampling estimates for standard deviation of GLMM estimation, $k_{n}=C[n^{1/3}]$, $C=1,2,4,8$.}
\label{tb: cov binom subsampling 2}%
\end{table}

\section{Alternative to Marginal Likelihood Estimation} \label{Sec: MGLM est}%

We close with a discussion of the alternative approach for binary time series regression modelling proposed by \cite{wu2014parameter}. Their modified GLM (MGLM) method replaces $\exp(x_{nt}^\T\beta)/(1+\exp(x_{nt}^\T\beta))$ in the GLM log-likelihood \ref{eq: loglikglm}  with a function $\pi(x_{nt}^\T\beta)$ representing the marginal mean to arrive at the objective function
\begin{equation} \label{eqn: MGLM objective function}
l_2(\beta) =\sum_{t=1}^n \left[ y_t \log \pi(x_{nt}^{\T}\beta) + (1-y_t)\log (1- \pi(x_{nt}^{\T}\beta)) \right], \quad \pi(x_{nt}^{\T}\beta)= \int \frac{e^{x_{nt}^{\T}\beta+\alpha_t}}{1+e^{x_{nt}^{\T}\beta+\alpha_t}} g(\alpha_t) d\alpha.
\end{equation}
The MGLM estimate $\hat\beta_2$ is found by iterating two steps starting with the GLM estimate of $\beta$: step 1, estimate the curve of $\pi(u)=\int e^{u+\alpha}/(1+e^{u+\alpha}) g(\alpha)d\alpha$ non-parametrically on $u\in \mathbb{R}$; step 2, maximize $l_2$ with respect to $\beta$, based on the estimate of $\pi(u)$ obtained in the first step. Steps 1 and 2 are repeated and the iteration stops when the maximum value of $l_2$ is reached and the last update of $\beta$ is then regarded as the MGLM estimator. Implementation details are provided in \cite{wu2014parameter}.

In defining their method \cite{wu2014parameter} do not require that $\pi(u)=\int e^{u+\alpha}/(1+e^{u+\alpha}) g(\alpha)d\alpha$ for any distribution $g$ of the latent process. Hence it is not required that $\pi(u)$ be non-negative and strictly increasing in $u$. However, their main theorem concerning consistency and asympototic normality of $\hat \beta_2$ is stated in terms of this latent process specification. For such specifications, application of their non-parameteric method for estimating $\pi(u)$ requires additional constraints which are not currently implemented. For example, taking the first derivative with respect to $u$, gives $\dot \pi(u)= \int e^{u+\alpha}/(1+e^{u+\alpha})^2 g(\alpha)d\alpha\le \pi(u)(1-\pi(u))$ so  $\dot\pi(u)\in [0,0.25]$ by application of Jensen's inequality. When applied to the Cambridge-Oxford Boat Race time series the non-parameteric estimate of $\pi(u)$ is not monotonic and produces marginal estimates, at values of $u$ between the gaps in the observed values of $x_{nt}^\T \hat\beta_2$, which are zero and therefore not useful for prediction at new values of the linear predictor.

Although not implemented in the R-code of \cite{wu2014parameter}, this constraint as well as that of monotonicity can be enforced in the nonparametric estimation of $p(u)$ using an alternative local linearization to that used in \cite{wu2014parameter}. For example, with constraints of monotonicity and $\dot \pi(u)\in [0,0.25]$, different estimates $\hat\beta_1=0.2093$, $\hat\beta_2=0.1899$ (compared to $\hat\beta_1=0.237$, $\hat\beta_2=0.168$ in \cite{wu2014parameter}) are observed for the model for the Cambridge-Oxford boat race series that they analyse. In this example, the marginal likelihood estimates give $\hat\tau=0$ and hence $\hat \beta$ degenerates to GLM estimate which differs from that of \cite{wu2014parameter}. Somehow the marginal method (with or without monotonocity contraints) is avoiding the degeneracy issue that arises with the marginal estimation method proposed in this paper. This needs to be further understood.

While MGLM is computationally much more intensive than using standard GLMM methods for obtaining the marginal estimates it appears to avoid degeneracy but for reasons that are not fully understood at this stage. Additionally, it is not clear the extent to which MGLM with or without the proper constraints implied by a latent process specification avoids the high proportion of degenerate estimates observed with GLMM for binary data. Additionally the extent to which MGLM reproduces the correct curve for the marginal probabilities $\pi(u)$ when the true data generating mechanism is defined in terms of a latent process (parameter driven specification) has not been investigated. The extent to which the MGLM estimate of $\pi(u)$ differs from the curve defined by a latent process specification might form the basis for a non-parametric test of the distribution of the latent process -- when $\{\alpha_t\}$ is not Gaussian, the true curve of $\pi(u)$ is not the same with that evaluated under GLMM fits, and we may end up into different results for the estimates of $\hat\beta_2$ and $\hat\beta$.

\section{Discussion}\label{Sec: discussion}

To overcome the inconsistency of GLM estimates of the regression parameters in parameter driven binomial models time series models we have proposed use of the marginal likelihood estimation, which can be easily conducted using the generalized linear mixing model fitting packages. We have shown that the estimates of regression parameters and latent process variation obtained from this method are consistent and asymptotically normal even if the observations are serially dependent. The distribution of the marginal estimates is required for a score test of serial dependence in the latent process something which we will report on elsewhere. The asymptotic results and proofs thereof have assumed that the latent process is Gaussian which has helped streamline the presentation. This is not required for all results except for Lemma 2 (asymptotic identifiability for the binary case) which relies directly on the normal distribution. The proofs can be readily modified provided we assume that the moment generating function $m_{\alpha_t}(u)$ of $\alpha_t$ is finite for all $u < \sqrt{(d_2)}$ where $\tau < d_2$ defines the parameter space.

The structure of the model considered here is such that the theoretical results apply to other response distributions such as the Poisson and negative binomial with very little change in the proofs of theorems. GLM estimation in these cases is consistent and asymptotically normal regardless of serial dependence in the latent process. The same will be true of the use of marginal estimation with the advantage that the latent process variability is also estimated. While we have not yet shown this, we expect that for these other response distributions the use of marginal estimation will lead to more efficient estimates of the regression parameters.

For all response distributions and for moderate sample sizes the marginal estimation method can result in a non-zero probability of $\hat \tau  = 0$. As we have observed in simulations, and explained via theoretical asymptotic arguments, this is particularly problematical for binary responses with very high probabilities being observed (and expected from theory) for `pile-up' probability for $\hat \tau$. We have observed that for binomial data ($m>1$) the `pile-up' probability quickly decreases to zero and, as a result of this observation, we anticipate that this probability will not typically be large for Poisson and negative Binomial responses.

For binary data we have developed a useful upper bound approximation to this probability and subsequently proposed an improved mixture distribution for $\hat\beta$ in finite samples. These theoretical derivations are well supported by simulations presented. While this mixture distribution cannot be used based on a single time series none-the-less it provided useful insights into the sampling properties of marginal estimation for binary time series. Additionally the derivations suggest that regression models in which $x^{\T} \beta$ varies over an interval over which the inverse logit function is approximately linear will be particularly prone to the `pile-up' problem and this persists even when there is strong serial dependence. Practitioners should apply the marginal likelihood method with caution in such situations.

\section{Acknowledgements}

We thank Dr Wu and Dr Cui for providing us with the R-code for their application of the MGLM method to the Boat Race Data reported in \cite{wu2014parameter}.

\section{Appendix: A} \label{Sec: Pf: A}

\begin{proof}[Proof of Lemma \ref{lem: Identifiable Binom}] \label{Pf: lem: Identifiable Binom}
We consider the deterministic regressors only in this proof but the same arguments can be extended to stochastic regressors. Now $M$ is the largest value of $m$ for which $k_M>0$, so $Q(\delta) = Q(\delta_0)$ if and only if \begin{equation*}
\int_0^1 \sum_{j=0}^M \pi^0(j,h(u))\left(\log \pi(j,h(u)) - \log \pi^0(j,h(u))
\right) du = 0
\end{equation*}

Since the integrand is non-positive the integrand can be zero if and only if the integrand is zero almost everywhere. Hence, for a contradiction, assume $\exists \delta\ne \delta_0$ such that
\begin{equation*}
\sum_{j=0}^M \pi^0(j,h(u))\left(\log \pi(j,h(u))- \log \pi^0(j,h(u))\right) = 0,
\quad \forall u \in [0,1].
\end{equation*}
and this can only happen if $\pi^0(j,h(u))=\pi(j,h(u))$ for all $j=0,\ldots,M$. Since
\[
\pi(j,h(u))=\int {{M}\choose {j}}\dot{b}(h(u)^{\T}\beta+\sqrt{\tau} z)^j (1-\dot{b}(h(u)^{\T}\beta+\sqrt{\tau} z))^{M-j}\phi(z)dz
\]
it is straightforward to show, by iterating from $j=0, \ldots, M$, that  $\pi^0(j,h(u))=\pi(j,h(u))$ for all $j=0,\ldots,M$ is equivalent to
\begin{equation}\label{eq: Binom identity Spec}%
\int \dot{b}(h(u)^{\T}\beta+\sqrt{\tau} z)^j\phi(z)dz = \int \dot{b}(h(u)^{\T}\beta_0+\sqrt{\tau_0} z)^j\phi(z)dz, \quad j=1,\ldots,M,
\end{equation}
for any $u\in [0,1]$

We next show that the only way this can hold is if $\delta=\delta_0$. Fix $u$ and denote $a = h(u)^{\T}\beta$, $a_0=h(u)^{\T}\beta_0$ and $\sigma=\sqrt{\tau}$,
\[
d_j(a,\sigma) = E\left[\dot{b}(a+\sigma z)^j\right] - E\left[\dot{b}(a_0+\sigma_0 z)^j\right],
\quad j=1,\ldots,M.
\]
where expectation is with respect to the density $\phi(\cdot)$. Hence \eqref{eq: Binom identity Spec} is equivalent to
so that $d_1(a_0,\sigma_0) = d_2(a_0,\sigma_0) = \cdots = d_M(a_0,\sigma_0) = 0$. Assume $\eta_0 = (a_0,\sigma_0)$, if there exists $\eta\ne \eta_0$ such that \eqref{eq: Binom identity Spec} holds, then there is $\Vert\eta^\ast-\eta_0\Vert \le \Vert\eta-\eta_0\Vert$ such that
\begin{equation}\label{Pf: lem: eq: JacMat}%
\begin{pmatrix}
d_1(a,\sigma) \\
d_2(a,\sigma) \\
\vdots\\
d_M(a,\sigma)
\end{pmatrix} = %
\begin{bmatrix}
E \left[\dot{b}(\eta^\ast)^{0}\ddot{b}(\eta^\ast)\right] &  E \left[ \dot{b}(\eta^\ast)^{0}\ddot{b}(\eta^\ast)z\right]\\
2 E \left[\dot{b}(\eta^\ast)^{1}\ddot{b}(\eta^\ast)\right] & 2 E \left[ \dot{b}(\eta^\ast)^{1}\ddot{b}(\eta^\ast)z \right]\\
\vdots & \vdots\\
M E \left[ \dot{b}(\eta^\ast)^{M-1}\ddot{b}(\eta^\ast)\right] &
M E \left[ \dot{b}(\eta^\ast)^{M-1}\ddot{b}(\eta^\ast)z\right]
\end{bmatrix}
\begin{pmatrix}
v_1 \\
v_2
\end{pmatrix} = J(\eta^\ast)\binom{v_1}{v_2}
\end{equation}
where $v_1=a-a_0$ and $v_2=\sigma-\sigma_0$ cannot be zero at the same time when
$\eta\ne \eta_0$, which is equivalent to the matrix $J(\eta^\ast)$ being of full rank. But $J(\eta^\ast)$ is not of full rank if and only if the ratios of the second column to the first column are the same for all $j=1,\ldots, M$. However, we now show that this ratio increases as $j$ increases.  Since $\dot{b}(\cdot)$ and $\ddot{b}(\cdot)$ are non-negative functions we can define
probability densities
\[
g_j(z) = \frac{\dot{b}(a^\ast +\sigma^\ast z)^{j-1}\ddot{b}(a^\ast+\sigma^\ast z) \phi(z)}{\int \dot{b}(a^\ast +\sigma^\ast z)^{j-1}\ddot{b}(a^\ast+\sigma^\ast z) \phi(z)d z}, \quad j=1, \ldots, M
\]
so that
\begin{equation*}
\frac{E\left[\dot{b}(\eta^\ast)^{j}\ddot{b}(\eta^\ast)z\right]}{E\left[ \dot{b}(\eta^\ast)^{j}\ddot{b}(\eta^\ast)\right]} = \frac{\int z \dot{b}(a^\ast+
\sigma^\ast z) g(z)d z}{\int \dot{b}(a^\ast+\sigma^\ast z) g(z)d z} = \frac{E_{g_j}(z \dot{b}(a^\ast+\sigma^\ast z))}{ E_{g_j}(\dot{b}(a^\ast+\sigma^\ast z))}
\end{equation*}
where $E_{g_j}()$ denotes expectation with respect to $g_j$.
But since $\dot{b}(a^\ast+\sigma^\ast z)$ is an increasing function of $z$, $z$ and $\dot{b}(a^\ast+\sigma^\ast z)$ are positively correlated. Therefore, \begin{equation*}
E_{g_{j}}(z)< E_{g_{j}}(z \dot{b}(a^\ast+\sigma^\ast z))/E_{g_{j}}(\dot{b}(a^\ast+\sigma^\ast z))
\end{equation*}
it follows that
\begin{equation*}
\frac{E\left[\dot{b}(\eta^\ast)^{j-1}\ddot{b}(\eta^\ast)z\right]}{E\left[ \dot{b}(\eta^\ast)^{j-1}\ddot{b}(\eta^\ast)\right]} < \frac{E\left[ \dot{b}(\eta^\ast)^{j}\ddot{b}(\eta^\ast)z\right]}{E\left[ \dot{b}(\eta^\ast)^{j}\ddot{b}(\eta^\ast)\right]}, \quad j=1,\ldots, M.
\end{equation*}
Then when \eqref{eq: Binom identity Spec} holds, \eqref{Pf: lem: eq: JacMat} has a unique solution of $(0,0)$ for $(v_1,v_2)$, which contradicts to the assumption that $\eta\ne \eta_0$. Thus
\eqref{eq: Binom identity Spec} holds if and only if $\eta=\eta_0$, which implies
$a=a_0$ and $\sigma=\sigma_0$. By Condition \ref{Cond: Reg Full Rank}, we can conclude that $a=a_0$ for all $u$ implies $\beta=\beta_0$. Therefore Condition \ref{Cond: Ident and Const} holds for
$M\ge 2$.
\end{proof}

\begin{proof}[Proof of Lemma \ref{lem: Identifiable Binary}]\label{Pf: lem: Identifiable Binary}%

This proof considers the Fourier transform method used in \cite{wang2003matching}. Assume $\exists \delta\ne \delta_0$ such that $\forall x\in \mathbb{X}$,
\begin{equation*}
\pi^0(1)=\int \dot{b}(x^{\T}\beta_0 + \sigma_0 z)\phi(z)dz = \int \dot{b}(x^{\T}\beta+ \sigma z)
\phi(z)dz =\pi(1), \quad \sigma=\sqrt{\tau}.
\end{equation*}
For connected $\mathbb{X}$, the first derivative with respect to $x$ of both sides are also the same, that is
\begin{equation}\label{eq: pi0 d1 to pi d1}%
\int \ddot{b}(x^{\T}\beta_0+\sigma_0 z)\phi(z)dz \cdot \beta_0 = \int \ddot{b}(x^{\T}\beta+
\sigma z)\phi(z)dz \cdot \beta
\end{equation}
which implies there exists a constant $c_1>0$ such that $c_1= \beta_{k}/\beta_{0,k}$ for $k=1,\ldots,r$, with $\eta=x^{\T}\beta$, $c_2=\sigma/\sigma_0 >0$, \eqref{eq: pi0 d1 to pi d1} can be rewritten as convolutions
\begin{equation*}
\int g(u)h(\eta-u)du = c_1 \int g(u)h(c_1\eta-c_2 u)du; \quad
h(x)= \frac{e^{x}}{(1+e^{x})^2}, \quad g(x)=\frac{1}{\sigma_0\sqrt{2\pi}}\exp(-\frac{x^2}{2\sigma_0^2}).
\end{equation*}
let $G(s)$ be the Fourier transform of the normal density $g(\cdot)$, $H(s)$ be the Fourier transform of the logistic density $h(\cdot)$, using the fact that the Fourier transform of the convolution is the product of Fourier transform of each function,
\begin{align*}
G(s)H(s)&=\int \left(\int c_1 \phi(u)h(c_1\eta-c_2 u)du\right) e^{-i\eta s}d\eta \\
&= c_1 \int \phi(u) \left(\int h(c_1 \eta - c_2 u) e^{-i\eta s}d\eta\right)du \\
&= c_1 \int \phi(u) \left(\int h(c_1 \eta - c_2 u) e^{-i(c_1\eta - c_2 u)(s/c_1)}
d(c_1\eta - c_2 u)\right)|c_1|^{-1}e^{-iu(c_2s/c_1)}du\\
&= \int \phi(u) e^{-iu(c_2s/c_1)}du H(s/c_1) = \frac{c_1}{|c_1|}G(c_2s/c_1)H(s/c_1),
\quad c_1>0.
\end{align*}
it follows that $G(s)H(s)= G(c_2s/c_1)H(s/c_1)$, $\forall s\in \mathbb{R}$. The Fourier transform for the mean zero normal distribution is $G(s) = \exp(-\frac{1}{2}\sigma_0^2 s^2)$, and the Fourier transform for the logistic distribution is $H(s)= 2\pi s/(e^{\pi s}- e^{-\pi s})$, thus for any $s\ne 0$,
\begin{equation*}
\exp(-\frac{1}{2}\sigma_0^2 s^2)\frac{1}{\sinh(\pi s)} = \exp(-\frac{1}{2}\sigma_0^2 s^2 \left(\frac{c_2}{c_1}\right)^2)\frac{1}{c_1\sinh(\pi s/c_1)}
\end{equation*}
for any fixed $c_1\ne 1$, $c_2$ can be expressed as a function of $s$ and this function is not a constant over $s$, which contradicts to the definition of $c_2$. Hence the equality holds if and only if $c_1=1$ and $c_2=1$.
\end{proof}

\section{Appendix: B} \label{Sec: Pf: B}

\begin{proof}[Proof of Theorem \ref{Thm: Consist&Asym Marginal Like Estimator}]
\label{Pf: Thm: Consist&Asym Marginal Like Estimator}

This proof is presented in three steps: first, we show, for any $\delta \in \Theta$, that $E(Q_n(\delta))$ defined in \eqref{eqn: E Qn delta} converges to $Q(\delta)$ defined in \eqref{eqn: lim Q delta Cond 2a} and \eqref{eqn: lim Q delta Cond 2b} under Condition 2a and 2b respectively; second, that $Q_n(\delta)- E(Q_n(\delta))\overset{\textrm{a.s}}{\to} 0 $ where $Q_n(\delta)$ is defined in \eqref{eqn: Qn delta} from which, using compactness of the parameter space, it follows that  $\hat\delta\overset{\textrm{a.s.}}\rightarrow \delta_0$; third, that
$\sqrt{n}(\hat\delta -\delta_0)\overset{\textrm{d}}\to N(0,\Omega_{1,1}^{-1}\Omega_{1,2}\Omega_{1,1}^{-1})$.

\noindent \textit{Proof that} $E(Q_n(\delta)) \overset{\textrm{a.s.}}{\to}Q(\delta)$: Use Jensen's inequality multiple times we have
\begin{align}\label{eq: Min ln pi}
& - \sum_{j=0}^{m_t}\pi_{t}^0(j)\ln \pi_{t}(j) \le\left(\sum_{j=0}^{m_t}\pi_{t}^0(j)\right) \left(\sum_{j=0}^{m_t}(-\ln \pi_{t}(j)) \right) = - \sum_{j=0}^{m_t}\ln \pi_{t}(j) \nonumber\\
\le & - \sum_{j=0}^{m_t} \left(j(x_{t}^T\beta) - m_t (\ln 2 + \max(x_{t}^T\beta + \tau/2,0)) + c(j) \right) \nonumber \\
\le & \sum_{j=0}^{m_t} \left[ m_t(\ln 2 + \tau/2)+ m_t\vert x_{t}^T\beta\vert - c(j)\right] < m_t(1+m_t)\left(\ln 2 + \tau/2 + \vert x_{t}^T\beta\vert\right)
\end{align}
then conditional on $m_t$ and $x_{t}$, $\sum_{j=0}^{m_t}\pi_{t}^0(j)\ln \pi_{t}(j)$ is bounded for all $t$ and $\delta$. Under Condition 2a, the regressor $x_{nt}:=h(t/n)$ is nonrandom as is the marginal density $\pi_{nt}$. Then the strong law of large numbers for mixing processes \citep{mcleish1975maximal} applied to $\{m_t\}$  gives
\begin{equation*}
\lim_{n\to\infty}\frac{1}{n}\sum_{t=1}^n \sum_{j=0}^{m_t}\pi_{nt}^0(j)\log \pi_{nt}(j) =
\lim_{n \to \infty} \frac{1}{n}\sum_{t=1}^n E\left[ \sum_{j=0}^{m_t} \pi_{nt}^0(j)\log \pi_{nt}(j)\right]=Q(\delta)
\end{equation*}
defined in \eqref{eqn: lim Q delta Cond 2a}.
For Condition 2b the ergodic properties of the stationary processes $m_t$ and $X_t$ can be used to establish
\begin{equation*}
\lim_{n\to\infty} \frac{1}{n}\sum_{t=1}^n \sum_{j=0}^{m_t}\pi_t^0(j)\log \pi_t(j)= \lim_{n\to\infty}\sum_{m=1}^M \frac{n_m}{n} \left(\frac{1}{n_m}\sum_{\{t:m_{t}=m\}}\sum_{j=0}^m \pi_{t}^0(j)\log \pi_{t}(j)\right)=
Q(\delta)
\end{equation*}
defined in \eqref{eqn: lim Q delta Cond 2b}.
\bigskip

\noindent \textit{Consistency}: We write \eqref{eqn: Qn delta} as $Q_n(\delta) = n^{-1} \sum_{t=1}^\infty q_t(\delta)$ where $q_{t}(\delta) = \log f(y_{t}|x_{nt},\delta)$.  By \citet[Proposition 1]{blais2000limit}, $\{q_t(\delta)\}$ is strongly mixing for any $\delta \in \Theta$. To apply the strong law of large numbers for mixing process in \cite{mcleish1975maximal} we need to show that $\exists \lambda \geq 0$ such that
\begin{eqnarray*}
\sum_{t=1} ^{\infty} \|q_{t}(\delta) - E q_{t}(\delta)\|_{2+\lambda}^{2}/t^{2} < \infty
\end{eqnarray*}
where $\|\cdot\|_p$ denotes the $L^p$ norm. By Minkowski's inequality and H\"{o}lder's inequality,
\begin{equation*}
\left\Vert q_{t}(\delta) - E q_{t}(\delta) \right\Vert_{2+\lambda} \le 2\left\Vert q_{t}(\delta)\right\Vert_{2+\lambda}
\end{equation*}
using similar derivations as used in \eqref{eq: Min ln pi},
$|q_{t}(\delta)| \le \vert - m_t |x_{nt}^T\beta| + c(y_t) - m_t(\ln 2+\tau/2) \vert$
where $m_t$ is bounded. It suffices to establish $\sum_{t=1}^\infty \Vert q_t(\delta)\Vert_{2+\lambda}^2/t^2< \infty$ if $\sum_{t=1}^\infty \vert x_{nt}^{\T} \beta\vert^2/t^2<\infty$. Under Condition 2a, $\{x_{nt}^{\T}\beta\}$ is bounded for any given $\beta$; $\sum_{t=1}^\infty 1/t^2< 2$, therefore the result follows. Under Condition 2b, we have $E\Vert x\Vert^2 <\infty$, given any $\beta$, for all $\varepsilon > 0$,
\begin{equation*}
P\left(\sum_{t=1}^n \vert x_t^{\T}\beta\vert^2/t^2 \ge \emph{K}\right) \le (2/\emph{K})
E\vert x^{\T}\beta\vert^2 \le \varepsilon, \quad \textit{ if } \quad \emph{K}\ge 2E\vert x^{\T}\beta\vert^2/\varepsilon.
\end{equation*}
Then $n^{-1}\sum_{t=1}^{n} \left[q_{t}(\delta) - E q_{t}(\delta)\right] \overset{a.s.} \rightarrow 0$ for any $\delta \in \Theta$. Together with the first part of the proof given above we now have $Q_n(\delta) \overset{\textrm{a.s}}{\to} Q(\delta)$. Since $\Theta$ is a compact set, and $Q_{n}(\delta)$ is a continuous function of $\delta$ for all $n$, by
\citet[Theorem 3.3]{gallant1988unified}, $\hat{\delta}: = \arg \underset{\Theta}\max~Q_n(\delta)\overset{a.s.} \rightarrow \delta_{0}.$
\bigskip

\noindent \textit{Asymptotic Normality}: Using a Taylor expansion,
\begin{equation*}
\sqrt{n}(\hat\delta - \delta_0)= -\left(\frac{1}{n}\sum_{t=1}^n \ddot{l}_t(\delta^\ast)\right) ^{-1} \frac{1}{\sqrt{n}}\sum_{t=1}^n \dot{l}_t(\delta_0), \quad \delta^\ast\overset{a.s.}\to \delta_0
\end{equation*}
the asymptotic normality of $\sqrt{n}(\hat\delta-\delta_0)$ can be obtained if
\[ -\frac{1}{n}\sum_{t=1}^n \ddot{l}_t(\delta^\ast) \overset{p}\to \Omega_{1,1} \quad \textrm{and} \quad \frac{1}{\sqrt{n}}\sum_{t=1}^n \dot{l}_t(\delta_0)\overset{d}\to N(0,\Omega_{1,2}). \]

Conditional on $m_t$ and $x_{nt}$, $\{\ddot{l}_t(\delta)\}$ is strongly mixing. Then by Chebyshev's inequality and \citet[Theorem 17.2.3]{ibragimov1971independent}, $\exists \epsilon>0$ such that
\[
\frac{1}{n}\sum_{t=1}^n \ddot{l}_t(\delta)- \frac{1}{n}\sum_{t=1}^n E(\ddot{l}_t(\delta)) \overset{p}\to 0, \quad \Vert \delta-\delta_0\Vert \le \epsilon
\]
Since $\delta^\ast\overset{a.s.}\to \delta_0$, and the continuity of $\ddot{l}_{t}(\cdot)$ with respect to $\delta$, $n^{-1}\sum_{t=1}^n \ddot{l}_{t}(\delta^\ast) - n^{-1}\sum_{t=1}^n \ddot{l}_t(\delta_0)\overset{a.s.}\to 0$, and $n^{-1}\sum_{t=1}^n \ddot{l}_t(\delta_0) - n^{-1}\sum_{t=1}^n E(\ddot{l}_t(\delta_0)) \overset{\textrm{a.s.}}{\to} 0$. Now $E(\ddot{l}_t(\delta_0))= E(\dot{l}_t(\delta_0)\dot{l}_t^{\T}(\delta_0))$, and hence it follows that $n^{-1}\sum_{t=1} \ddot{l}_t(\delta^\ast) \overset{p}\to \Omega_{1,1}$.

Next we show that $\Omega_{1,1}$ is positive definite. Let $s=(s_{1},s_{2})$ be an $(r+1)$ dimensional constant vector, without lose of generality, $s^{\T} s= 1$. Define $q_{t}(\delta_0)=s^{\T}\dot{l}_t(\delta_0)$, note $\det(\Omega_{1,1})\ge 0$ and $\det(\Omega_{1,1})=0$ only if $E\left(q_t^2(\delta_0)\vert m_t, x_{nt} \right)=0$ for all $t$.
Then under Condition \ref{Cond: Reg Full Rank}, $\Omega_{1,1}$ is positive definite. The limit, $\Omega_{1,1}$, under Condition \ref{Cond: Reg Trend Type}a is given in \eqref{eq: PD InfMat Lim}. For stationary regressors such limit can be easily obtained using ergodic theorem.

Next we show $\Omega_{1,2}$ exists. Note
\begin{equation*}
\Omega_{1,2}=\underset{n\to\infty}\lim \Var\left(\frac{1}{\sqrt{n}}\sum_{t=1}^n q_{t}(\delta_0) \right) = \sum_{h=0}^{n-1}\left(\frac{1}{n}\sum_{t=1}^{n-h} \Cov(q_t(\delta_0), q_{t+h}(\delta_0))\right) + \sum_{h=1}^{n-1}\left(\frac{1}{n}\sum_{t=h+1}^n \Cov(q_{t}(\delta_0), q_{t-h}(\delta_0))\right)
\end{equation*}
then $\Omega_{1,2}$ exists if
\[
\underset{n\to\infty}\lim \sum_{h=0}^{n-1}\left(\frac{1}{n}\sum_{t=1}^{n-h} |\Cov(q_{t}(\delta_0), q_{t+h}(\delta_0)) | \right) < \infty, \quad \underset{n\to\infty}\lim \sum_{h=1}^{n-1} \left(\frac{1}{n}\sum_{t=h+1}^n |\Cov(q_{t}(\delta_0), q_{t-h}(\delta_0))|\right) < \infty.
\]
Since the $\{q_t(\delta_0)\}$ is strong mixing, by Theorem 17.3.2 in \cite{ibragimov1971independent},
\begin{equation*}
\sum_{h=0}^{n-1}\left(\frac{1}{n}\sum_{t=1}^{n-h} \left\vert \Cov(q_t(\delta_0), q_{t+h}(\delta_0)) \right\vert \right) \le 2 \sum_{h=0}^{n-1} \nu(h)^{\lambda/(2+\lambda)}W_h < \infty,
\end{equation*}
where \begin{equation*} W_h = \lim_{n\to\infty} \frac{1}{n}\sum_{t=1}^{n-h}\left[ 4+ 3(c_t c_{t+h}^{1+\lambda} + c_t^{1+\lambda}c_{t+h}) \right], \quad c_t\ge \|q_{t}(\delta_0)\|_{2+\lambda} \end{equation*} Such $c_t$ exists, for example, take $\lambda=2$, use Cauchy-Schwarz's inequality,
\begin{align*}
E\left( |q_{t}(\delta_0)| ^{4}|m_t,x_{nt}\right) & \le E\left\{ \int f(y_{t}|x_{nt},z_{t},\delta_0) \phi(z_{t})\left((y_{t}-m_t\dot{b}(W_{0,t}))(s_{1}^{\T}x_{nt}+ s_{2}z_{t})\right)^{4} dz \cdot f^{-1}(y_{t}|x_{nt},\delta_0)\right\}\\
& = E\left[ m_t\ddot{b}(W_{0,t})(1+ (3m_t-6)\ddot{b}(W_{0,t}))(s_{1}^{\T}x_{nt}+ s_{2}z_{t})^{4} | m_t, x_{nt}\right]
\end{align*}
Then by the application of CLT for mixing process in Theorem 3.2, \cite{davidson1992central} we will have \newline $n^{-1/2}\sum_{t=1}^n \dot{l}_t(\delta_0)\overset{d}\to N(0,\Omega_{1,2})$.
\end{proof}

\begin{proof}[Proof of Theorem \ref{Thm: GLM asymptotics}] \label{Pf: Thm: GLM asymptotics}
Following the proof of \cite{davis2000autocorrelation} and \cite{wu2014parameter}, let $u = \sqrt{n}(\beta - \beta ^\prime)$ (note here we centre on $\beta'$ and not the true value $\beta_0$ as was done in these references). Then maximizing
\begin{equation*}
l_n(\beta)= \sum_{t=1}^n\left[ y_t(x_{nt}^{\T}\beta)- m_t b(x_{nt}^{\T}\beta) + c(y_t)\right]
\end{equation*}
over $\beta$ is equivalent to minimizing $g_n(u)$ over $u$ where $g_{n}(u):= -l_{n}(\beta^\prime + u/\sqrt{n}) + l_{n}(\beta^{\prime})$. Let $\hat{u}=\arg \min \underset{n\to\infty}\lim g_n(u)$. Write $g_{n}(u): = B_{n}(u) - A_{n}(u)$ where
\[
B_{n}(u): =\sum_{t=1} ^{n} m_{t} \left( b(x_{t}^{T}\beta^\prime + x_t^T u/\sqrt{n})- b(x_{t}^{T}\beta^{\prime}) - \dot{b}(x_{t}^{T}\beta^{\prime})x_{t}^{T} u/\sqrt{n}\right)
\]
and
\[
A_{n}(u): = u^T \frac{1}{\sqrt{n}}\sum_{t=1}^{n}(y_{t} - m_{t}\dot{b}(x_{t}^{T}\beta^{\prime}))
x_{t} = u^T U_n.
\]
Using similar procedures as in the proof of Theorem 1. in \cite{wu2014parameter} it is straightforward to show
\[
B_{n}(u) \rightarrow \frac{1}{2}u^{T} \Omega_{1} u
\]
and
\[
E\left( e^{is^{T}U_{n}}\right) = \exp\left[-\frac{1}{2}s^T\Omega_{2}s\right].
\]
for each $u$. Since $g_{n}(u)$ is a convex function of $u$ and $\hat{u}_n$ minimizes $g_n(u)$, then an application of the functional limit theory gives $\hat{u}_n\overset{d}\to \hat{u}$, where $\hat{u}=\arg \min \underset{n\to\infty}\lim g_n(u)$. In conclusion,
\begin{equation*}
g_n(u)\overset{d}\to g(u)= \frac{1}{2}u^T \Omega_1 u - u^T N(0,\Omega_2)
\end{equation*}
on the space $C(\mathbb{R}^r)$, and $\hat{u}_n\overset{d}\to \hat{u}$, where $\hat{u}\sim
N(0, \Omega_{1}^{-1}\Omega_{2}\Omega_{1}^{-1}).$
\end{proof}

\begin{proof}[Proof of Theorem \ref{Thm: score dist}] \label{Pf: Thm: score dist}%

Since $S_{1,n}(\tilde\beta)$, a linear function of $U_{1,n}$ and $U_{2,n}$, to show that $\sqrt{n}\left(S_{1,n}(\tilde\beta)- E(S_{1,n}(\beta^\prime))\right)$ is normally distributed it is sufficient to show that the the joint distribution of $(U_{1,n}, U_{2,n})$ is multivariate normal. As defined in \eqref{eq: U1 U2},
\[
U_{1,n}:= S_{1,n}(\beta^\prime) - E\left(S_{1,n}(\beta^\prime)\right)= \frac{1}{\sqrt{n}} \sum_{t=1}^{n} e_{t,\beta^\prime} ^2 - E e_{t,\beta'}^2; \quad
U_{2,n}:= \Omega_{1}^{-1}\frac{1}{\sqrt{n}}\sum_{t=1}^{n} e_{t,\beta^\prime}x_{nt}= \frac{1}{\sqrt{n}}\sum_{t=1}^{n} e_{t,\beta^\prime}c_{nt}
\]
where $c_{nt} = \Omega_{1}^{-1}x_{nt}$, defines a sequence of non-random vectors.

Let $U_{nt}= (U_{1,nt}, U_{2,nt})$ be the joint vector at time $t$, then each dimension of $\{U_{nt}\}$ is uniformly bounded, strongly mixing and $E(U_{nt})=0$. We need to prove that $a_{1} U_{1,n} + a_{2}^T U_{2,n}$ has normal distribution for arbitrary constant vector $a = (a_1, a_2)$ where $a^T a =1$ without loss of generality. By the SLLN for mixing process in  \cite{mcleish1975maximal}, there exists a limiting matrix $\Omega_{U}$ such that
\begin{equation*}
\Var(\sum_{t=1}^{n} a^T U_{nt}) = \sum_{h=0}^{n-1}\left(\sum_{t=1}^{n-h} a^T\textrm{Cov}
(U_{nt}, U_{n,t+h}) a\right) + \sum_{h=1}^{n-1}\left(\sum_{t=1+h}^{n} a^T\textrm{Cov}(U_{nt},
U_{n, t-h}) a \right) \rightarrow  a^T \Omega_{U} a.
\end{equation*}
Then conditions of the CLT of \cite{davidson1992central} are satisfied, and we have $\sum_{t=1}^{n} U_{nt} \overset{d}\to N(0,\Omega_U)$.
\end{proof}

\begin{proof}[Proof of Theorem \ref{Thm: Prob mixdensity sigma=0}]

This proof follows \cite{moran1971maximum}. $\delta_0$ is the true value of the parameters, $\delta'=(\beta',0)$ is the limit of parameters that maximize
\eqref{eq: marginal log likelihood} for fixed $\tau=0$; $\hat\delta$ is the maximum likelihood estimators of \eqref{eq: marginal log likelihood}.

Consider first the distribution of $\hat\delta$ if $\hat\tau>0$. Since the unconditional joint distribution of $\hat\beta$ and $\hat\tau$ is multivariate normal with $N(0, \Omega_{1,1}^{-1} \Omega_{1,2}\Omega_{1,1}^{-1})$ -- see proof of Theorem \ref{Thm: Consist&Asym Marginal Like Estimator}, hence $F_2(c,\delta_0)$, the distribution of $\sqrt{n}(\hat\beta-\beta_0)|\hat\tau>0$ is skew normal based on $N(0,\Omega_{1,1}^{-1}\Omega_{1,2}\Omega_{1,1}^{-1})$.

When $\hat\tau=0$, use a Taylor expansion to the first derivatives around $\delta'$, then
\begin{equation*}
\sqrt{n}(\hat\beta -\beta') = \left(-\frac{1}{n}\sum_{t=1}^n \ddot{l}_{t}(\delta') \right)^{-1} \frac{1}{\sqrt{n}}\sum_{t=1}^n \frac{\partial l_{t}(\delta')}{\partial \beta} + o_p(1),
\end{equation*}
conditional on $n^{-1/2}\sum_{t=1}^n \partial l_{t}(\delta')/\partial \tau < 0$. As $n\to\infty$,
\[
\frac{1}{\sqrt{n}}\sum_{t=1}^n \frac{\partial l_{t}(\delta')}{\partial \delta} = \begin{pmatrix}
\frac{1}{\sqrt{n}}\sum_{t=1}^n e_{t,\beta'}x_{nt} \\
\frac{1}{2\sqrt{n}}\sum_{t=1}^n \left[ e_{t,\beta'}^2- m_t \ddot{b}(x_{t}^T\beta')\right]
\end{pmatrix} \overset{d}\to N(\begin{pmatrix} 0\\
E(S_{1,n}(\beta'))/2 \end{pmatrix}, \begin{pmatrix}
\Omega_{2} & K_{S}/2\\
K_{S}^T/2 & V_{S}/4 \end{pmatrix} )
\]
\end{proof}

\bibliography{MarginalLikelihoodBinomialTSarXiv}
\bibliographystyle{imsart-nameyear}

\end{document}